\documentclass[oneside,10pt]{article}          
\usepackage[b5paper]{geometry}	    
\usepackage{amsfonts,amsmath,latexsym,amssymb} 
\usepackage{amsthm}                
\usepackage{mathrsfs,upref}         
\usepackage{mathptmx}		    
\usepackage{fdc}	            

\usepackage[numbers]{natbib}
\usepackage{microtype}
\usepackage{dsfont}
\usepackage{graphicx}

\usepackage{oldgerm}
\DeclareFontFamily{U}{yswab}{}
\DeclareFontShape{U}{yswab}{m}{n}{<-> yswab}{}

\usepackage{bm}                	  

\newtheorem{theorem}{Theorem}[section]
\newtheorem*{theorem*}{Theorem}
\newtheorem{proposition}{Proposition}[section]

\newtheorem{corollary}{Corollary}

\theoremstyle{definition}

\newtheorem{definition}{Definition}
\newtheorem{example}{Example}

\newtheorem{remark}{Remark}

\begin{document}

\title[Properties of Prabhakar-type Operators]{Some Properties of Prabhakar-type Fractional Calculus Operators}


\author{Federico Polito, \v{Z}ivorad Tomovski}

\address{Federico Polito, Dipartimento di Matematica \emph{G.~Peano}, Universit\`a degli Studi di Torino, Italy\\
\email{federico.polito@unito.it}}

\address{\v{Z}ivorad Tomovski, Department of Mathematics, Sts.\ Cyril and Methodius University, Skopje, Macedonia, Department of Mathematics, University of Rijeka, Croatia}


\CorrespondingAuthor{Federico Polito}


\date{22.09.2015}                               




\begin{abstract}

		In this paper we study some properties of the Prabhakar integrals and derivatives and of some of their extensions such as
	the regularized Prabhakar derivative or the Hilfer--Prabhakar derivative. Some Opial- and Hardy-type inequalities are derived.
	In the last section we point out on some relationships with probability theory.
			
%

\end{abstract}

\maketitle



	\section{Introduction and background}
	
		The aim of this note is to study the properties of some integral and differential operators that can be related to
		a specific convolution-type integral operator, called Prabhakar integral, introduced by \citet{Prab}.
		Before describing the results obtained we start by recalling the basic definitions and mathematical tools
		that will be useful in the following.

		First, let us give to the reader some insights on the classical operators related to fractional calculus. For more in-depth information,
		applications and related topics
		it is possible to consult some of the classical references, e.g.\ \citet{MR1347689}, \citet{MR1658022},
		\citet{MR2218073}, \citet{MR2680847}, \citet{MR3244285}, \citet{MR2676137}.
		We also suggest the reader to consult the references listed in these books
		as the relevant literature is becoming richer and richer.

		In the following we give the definitions of the Riemann--Liouville integral which in some sense generalizes the classical
		multiple integral, and of its naturally associated differential operator, i.e.\ the Riemann--Liouville derivative.
		We then proceed by describing the regularized version of the Riemann--Liouville derivative, the so-called Caputo or
		Caputo--D{\v{z}}rba{\v{s}}jan derivative, introduced independently in the sixties by \citet{caputo2,caputo1969elasticita}
		and \citet{MR0224984}.
		
		\begin{definition}[Riemann--Liouville integral]
		    Let $f \in L^1_{\text{loc}}(a,b)$, $-\infty \le a < t < b \le \infty$, be a locally integrable real-valued function.
		    Let us further define the power-law kernel $\mathcal{K}_\alpha (t) = t^{\alpha-1}/\Gamma(\alpha)$, $\alpha>0$. The operator
		    \begin{align}
		        \label{rlint}
		        (I^{\alpha}_{a^+} f)(t) & =\frac{1}{\Gamma(\alpha)}\int_{a}^t\frac{f(u)}{(t-u)^{1-\alpha}}\textup{d}u
		        = (f \ast \mathcal{K}_\alpha )(t), \qquad \alpha > 0,
		    \end{align}
		    is called Riemann--Liouville integral of order $\alpha$.
		\end{definition}

		\begin{definition}[Riemann--Liouville derivative]
			\label{rie}
		    Let $f \in L^1(a,b)$, $-\infty \le a < t < b \le \infty$, and $f
		    \ast \mathcal{K}_{m-\alpha} \in W^{m,1}(a,b)$, $m = \lceil \alpha \rceil$,
		    $\alpha>0$,
		    where $W^{m,1}(a,b)$ is the Sobolev space defined as
		    \begin{align}
		        W^{m,1}(a,b) = \left\{ f \in L^1(a,b) \colon \frac{\textup{d}^m}{\textup{d}t^m} f \in L^1(a,b) \right\}.
		    \end{align}
		    The Riemann--Liouville derivative of order $\alpha$ is defined as
		    \begin{align}
		        \label{rlder}
		        ( D^\alpha_{a^+}f )(t) = \frac{\textup{d}^m}{\textup{d}t^m}I_{a^+}^{m-\alpha}f(t) = \frac{1}{%
		        \Gamma(m-\alpha)} \frac{\textup{d}^m}{\textup{d}t^m} \int_{a}^t (t-s)^{m-1-\alpha}f(s) \, \textup{d}s.
		    \end{align}
		\end{definition}

		We denote by $AC^{n}\left(a,b\right)$, $n\in \mathbb{N}$, the
		space of real-valued functions $f\left( t\right) $ with
		continuous derivatives up to order $n-1$ on $\left( a,b\right) $
		such that $f^{\left(n-1\right) }\left(t\right)$ belongs to the space of absolutely continuous functions
		$AC\left(a,b\right)$, that is,
		\begin{equation}
			AC^{n}\left(a,b\right) =\left\{ f:\left(a,b\right) \rightarrow
			\mathbb{R}\colon\frac{\textup{d}^{n-1}}{\textup{d}x^{n-1}}f \left( x\right) \in AC\left(
			a,b\right) \right\} .
		\end{equation}

		\begin{definition}[Caputo derivative or regularized Riemann--Liouville derivative]
			\label{capu}
		    Let the parameter $\alpha>0$, $m = \lceil \alpha \rceil$, and $f \in AC^m(a,b)$.
		    The Caputo derivative (also known as regularized Riemann--Liouville derivative) of order $\alpha>0$ is defined as
		    \begin{equation}
		        \label{Capu}
		        ({}^C D^{\alpha}_{a^+}f)(t)= \left(I_{a^+}^{m-\alpha}\frac{\textup{d}^m}{\textup{d}t^m}f\right)(t)= \frac{1%
		        }{\Gamma(m-\alpha)}\int_a^{t}(t-s)^{m-1-\alpha}\frac{\textup{d}^m}{\textup{d}s^m}f(s) \, \textup{d}s.
		    \end{equation}
		\end{definition}
		
		The above Definition \ref{capu} should be compared with the non-regularized case of Definition \ref{rie}. The reader should
		also be aware of the fact that the above derivatives can be defined in different ways. For example a more
		interesting and intuitive way of defining the Caputo derivative is given by the following theorem.

		\begin{theorem}
		    \label{gianduia}
		    For $f \in AC^m(a,b)$, $m=\lceil \alpha \rceil$, $\alpha\in\mathbb{R}^+\backslash\mathbb{N}$,
		    the Riemann--Liouville derivative of order $\alpha$ of $f$
		    exists almost everywhere and it can be written in terms of Caputo derivative as
		    \begin{align}
		        \label{perepe}
		        (D_{a^+}^\alpha f)(t) = ({}^C D_{a^+}^\alpha f)(t) + \sum_{k=0}^{m-1} \frac{(x-a)^{k-\alpha}}{\Gamma(k-\alpha+1)} f^{(k)}(a^+).
		    \end{align}
		\end{theorem}

		Theorem \ref{gianduia} is interesting in that it in practice describes the set of functions with which the
		Riemann--Liouville derivative can be regularized. It is well-known that if
		$f(t)\in AC^m(a,b)$,
		\begin{align}
		    \label{purupu}
		    \lim_{t \to a^+} \left(\frac{\textup{d}^k}{\textup{d}t^k} (I_{a^+}^{m-\alpha} f)\right)(t) = 0, \qquad 0 \le k \le m-1.
		\end{align}
		By taking the Laplace transform of both sides of \eqref{perepe} we immediately realize that
		\eqref{perepe} still holds if \eqref{purupu} is true.

		In the recent years new alternative definitions of fractional operators have been introduced in the literature. An
		interesting example is the so-called Hilfer derivative \citep{hilfer2008threefold, MR2517679, MR2739389}.
		The idea behind the introduction of this derivative is to interpolate between the Riemann--Liouville
		and the Caputo derivatives. As it is clear from the definition below, the Hilfer derivative depends on the
		parameter $\nu \in [0,1]$ that balances the individual contributions of the two fractional derivatives.

		\begin{definition}[Hilfer derivative]
		    Let $\mu\in (0,1)$, $\nu \in[0,1]$, $f \in L^1[a,b]$,
		    $-\infty \le a < t < b \le \infty$, $f \ast \mathcal{K}_{(1-\nu)(1-\mu)} \in AC^1[a,b]$,
		    where $\mathcal{K}_\alpha (t) = t^{\alpha-1}/\Gamma(\alpha)$. The Hilfer derivative is defined as
		    \begin{equation}
		        \label{hil}
		        (D^{\mu,\nu}_{a^+}f)(t)=\left(I_{a^+}^{\nu(1-\mu)} \frac{\textup{d}}{\textup{d}t}%
		        (I_{a^+}^{(1-\nu)(1-\mu)}f)\right)(t),
		    \end{equation}
		\end{definition}
		Notice that Hilfer derivatives coincide with Riemann--Liouville derivatives for $\nu=0$ and with
		Caputo derivatives for $\nu=1$.

		In order to proceed with the description of the relevant operators involved, we need to
		consider now the function			
		\begin{align}
			\label{efun}
			e^{\gamma}_{\rho,\mu,\omega}(t) = t^{\mu-1}E^{\gamma}_{\rho,\mu}\left(\omega t^{\rho} \right),
			\qquad t \in \mathbb{R}, \: \rho, \mu, \omega, \gamma\in \mathbb{C}, \: \Re(\rho)>0, 
		\end{align}
		where
	    \begin{equation}
	        E^{\gamma}_{\rho,\mu}(x)=\sum_{k=0}^{\infty}\frac{\Gamma(\gamma+k)}{%
	        \Gamma(\gamma)\Gamma(\rho k+\mu)}\frac{x^k}{k!},
	    \end{equation}
	    is the generalized Mittag--Leffler function first investigated by %
	    \citet{Prab}.			
		The so-called Prabhakar integral is constructed in a similar way of Riemann--Liouville integrals.
		The main difference is that, the power-law kernel $\mathcal{K}_\alpha(t)$
		in the integral representation of the operator is replaced by the function \eqref{efun}.
		The kernel \eqref{efun} actually generalizes $\mathcal{K}_\alpha(t)$ in the sense that
		$e^{0}_{\rho,\mu,\omega}(t) = \mathcal{K}_\mu(t)$.
		The Prabhakar integral is hence defined as follows \citep{Prab,Kil}.

		\begin{definition}[Prabhakar integral]
		    Let $f \in L^1(a,b)$, $0 \le a < t < b \le \infty$. The Prabhakar integral is defined as
		    \begin{equation}
		        \label{pra}
		        (\mathbf{E}^{\gamma}_{\rho,\mu, \omega, a^+}f)(t)
		        = (f \ast e^{\gamma}_{\rho,\mu,\omega} )(t)
		        = \int_a^{t}(t-y)^{\mu-1}E^{\gamma}_{\rho,\mu}\left[\omega (t-y)^{\rho} %
		        \right]f(y) \, \mathrm dy,
		    \end{equation}
		    where $\rho, \mu, \omega, \gamma\in \mathbb{C}$, with
		    $\Re(\rho),\Re(\mu)>0$.
		\end{definition}
		As just remarked, for $\gamma=0$, the integral \eqref{pra} coincides with \eqref{rlint}.
		
		An interesting property of the Prabhakar integral is the following \citep[][formula (2.21)]{Kil}.
		
		\begin{proposition}
			\label{calcolo}
			Let $\gamma,\rho,\mu,\eta,\sigma,\omega \in \mathbb{C}$, $\Re(\rho),\Re(\mu),\Re(\eta)>0$, $t \in \mathbb{R}$. Then
			\begin{align}
				(\mathbf{E}^{\gamma}_{\rho,\mu, \omega, 0^+} e^\sigma_{\rho,\eta,\omega})(t) = e^{\gamma+\sigma}_{\rho,\mu+\eta,\omega}(t).
			\end{align}
		\end{proposition}		
		
		Analogously to the classical fractional operators, a related differential operator can be defined as follows.
				
		\begin{definition}[Prabhakar derivative]
		    Let $f \in L^1(a,b)$, $0 \le a < t < b \le \infty$, and $f \ast e_{\rho, m-\mu,\omega}^{-\gamma}(\cdot) \in W^{m,1}(a,b)$,
		    $m = \lceil \mu \rceil$.
		    The Prabhakar derivative is defined as
		    \begin{align}
		    	\label{prabinv}
		        (\mathbf{D}^{\gamma}_{\rho, \mu, \omega, a^+}f)(t)
		        = \left(\frac{\mathrm d^m}{\textup{d}t^m}(\mathbf{E}^{-\gamma}_{\rho, m-\mu, \omega, a^+}f)\right)(t),
		    \end{align}
		    where $\mu,\omega,\gamma,\rho \in \mathbb{C}$, $\Re(\mu),\Re(\rho)>0$.
		\end{definition}

		The inverse operator \eqref{prabinv} of the Prabhakar integral, for $\gamma=0$,
		generalizes the Riemann--Liouville derivative as $E^0_{\rho,m-\mu}[\omega(t-y)^\rho]=1/\Gamma(m-\mu)$ in
		the kernel of \eqref{prabinv}. See also \citet{Kil} for a different but equivalent definition.
		
		The analogous operator to the Caputo derivative, that is the regularized Prabhakar derivative plays an
		important role in the construction of meaningful initial-value problems as it was noted in
		\citet{MR3239686}. The regularized Prabhakar derivative was introduced in \citet{d2013fractional}.
				
		\begin{definition}[Regularized Prabhakar derivative]
			Consider $f \in AC^m(a,b)$, $0 \le a < t < b \le \infty$. The regularized
			Prabhakar derivative reads
			\begin{align}
				\label{capl}
				({}^C\mathbf{D}^{\gamma}_{\rho, \mu, \omega,
				a^+}f)(t)&=\left( \mathbf{E}^{-\gamma}_{\rho, m-\mu, \omega, a^+}\frac{\textup{d}^m}{\textup{d}t^m}
				f\right)(t) \\
				&=(\mathbf{D}^{\gamma}_{\rho, \mu, \omega,
				a^+}f)(t)- \sum_{k=0}^{m-1} t^{k-\mu} E^{-\gamma}_{\rho,k-\mu+1}(\omega t^{\rho}) f^{(k)}(a^+). \notag
			\end{align}
		\end{definition}
		For $\gamma=0$ the operator \eqref{capl} coincides with the Caputo derivative \eqref{Capu}.
		\begin{remark}
		    Let $\mu>0$ and $f \in AC^m(a,b)$, $0 \le a < t < b \le \infty$. Then
		    \begin{align}
		        ({}^C\mathbf{D}^{\gamma}_{\rho, \mu, \omega,a^+}f)(t) = ( \mathbf{D}^{\gamma}_{\rho, \mu, \omega,a^+}
		        h ) (t),
		    \end{align}
		    where $h(t) = f(t) - \sum_{k=0}^{m-1}\frac{t^k}{k!}f^{(k)}(a^+)$.
		\end{remark}

		The last operator we need to derive the results described in Section \ref{mai} is similar to the Hilfer
		derivative but based on Prabhakar operators. Therefore, let us give the following
		
		\begin{definition}[Hilfer--Prabhakar derivative]
			\label{ge}
		    Let $\mu\in (0,1)$, $\nu \in[0,1]$, and let $f \in L^1(a,b)$,
		    $0 \le a < t < b \le \infty$, $f \ast e_{\rho, (1-\nu)(1-\mu),\omega}^{-\gamma(1-\nu)}(\cdot) \in AC^1(a,b)$.
			The Hilfer--Prabhakar derivative is defined by
			\begin{equation}
				\label{hilg}
				(\mathcal{D}^{\gamma, \mu, \nu}_{\rho, \omega, a^+} f)(t) =\left(\mathbf{E}%
				_{\rho, \nu(1-\mu),\omega, a^+}^{-\gamma \nu}\frac{\mathrm d}{\mathrm dt}( \mathbf{E}_{\rho,
				(1-\nu)(1-\mu),\omega, a^+}^{-\gamma(1-\nu)}f)\right)(t),
			\end{equation}
			where $\gamma, \omega \in \mathbb{R}$, $\rho > 0$, and where $(\mathbf{E}_{\rho, 0,\omega, a^+}^0 f)(t) = f(t)$.
		\end{definition}
		This Hilfer--Prabhakar derivative interpolates the two Prabhakar-type operators \eqref{prabinv} and \eqref{capl}
		and it specializes to the Hilfer derivative for $\gamma=0$.
		
		We can also think of a regularized version of \eqref{hilg}, that is, for $f \in AC^1(a,b)$, we have
		\begin{align}
			\label{regn}
			({}^C\mathcal{D}^{\gamma, \mu}_{\rho, \omega, a^+} f)(t) & =\left(\mathbf{E}_{\rho,
			\nu(1-\mu), \omega, a^+}^{-\gamma\nu} \left(\mathbf{E}_{\rho, (1-\nu)(1-\mu),\omega,
			a^+}^{-\gamma(1-\nu)} \frac{\textup{d}}{\textup{d}t}f\right) \right)(t) \\
			& = \left(\mathbf{E}_{\rho,
			1-\mu,\omega, a^+}^{-\gamma} \frac{\textup{d}}{\textup{d}t}f\right)(t). \notag
		\end{align}
		Notice that in the regularized Hilfer--Prabhakar derivative \eqref{regn}
		there is no dependence on the interpolating parameter $\nu$.
		
		Before proceeding to the results section we present here an estimate for the function
		$e_{\alpha ,\beta, \omega }^{\gamma }\left( t\right)$ that proves to be necessary in the proof of some of the theorems below.
		We refer in particular to Theorem 3 of \citet{tom}.
		
		\begin{theorem}[Tomovski--Pog\'any--Srivastava]
			\label{sta}
			For all $\alpha \in (0,1)$, $\gamma,\omega>0$, $\alpha \gamma > \beta-1>0$, the following uniform bound holds true:
			\begin{align}
				\label{poga}
				\left| e_{\alpha ,\beta, \omega }^{\gamma }(t) \right| \le \frac{\Gamma \left( \gamma-\frac{\beta-1}{\alpha} \right)
				\Gamma \left( \frac{\beta-1}{\alpha} \right)}{
				\pi \alpha \omega^{(\beta-1)/\alpha} \Gamma(\gamma) \left( \cos(\pi \alpha /2) \right)^{\gamma-(\beta-1)/\alpha} },
				\qquad t > 0.
			\end{align}
		\end{theorem}
		The proof of Theorem \ref{sta} makes use of an estimate of the Wright function given by \citet{MR0280762}.

	\section{Some bounds and operational calculus with Prabhakar-type operators}
	
		\label{mai}
		The first two results we present concern the boundedness of the Prabhakar integral \eqref{pra} in the space $L^p$, $p \in (0,1]$
		and in $L^1$ for $L^p$ functions, $p\in (1,\infty)$.

		\begin{theorem}
			\label{primoteo}
			Let $\alpha \in \left( 0,1\right)$, $\gamma ,\omega >0$, and $\alpha \gamma
			>\beta -1>0$. If $\varphi \in L^{p}\left( a,b\right)$, $0<p\leq
			1$, then the integral operator $\mathbf{E}_{\alpha ,\beta, \omega, a^+
			}^{\gamma}$ is bounded in $L^{p}\left( a,b\right) $ and 
			\begin{equation}
				\left\Vert (\mathbf{E}_{\alpha ,\beta, \omega, a^+
				}^{\gamma}\varphi)
				\right\Vert _{p}\leq M\left\Vert \varphi \right\Vert _{p},
			\end{equation}%
			where the constant $M$, $0<M<\infty$, is given by%
			\begin{equation}
				\label{givenby}
				M=\frac{\textup{Be} \left( \gamma -\frac{\beta -1}{\alpha }, 
				\frac{\beta -1}{\alpha }\right) }{\pi \alpha \omega ^{\frac{\beta -1}{\alpha 
				}} \left[ \cos \left( \frac{\pi \alpha }{2}%
				\right) \right] ^{\gamma -\frac{\beta -1}{\alpha }}}\left( b-a\right) ^{1/p},
			\end{equation}
			in which $\textup{Be}(\mu,\nu)$ is the Beta function.
		\end{theorem}

		\begin{proof}
			In order to prove the result it is sufficient to show that 
			\begin{equation}
				\left\Vert (\mathbf{E}_{\alpha ,\beta, \omega, a^+}^{\gamma}\varphi)
				\right\Vert _{p}^{p}=\int_{a}^{b}\left|
				\int_{a}^{x}\left( x-t\right) ^{\beta -1}E_{\alpha ,\beta }^{\gamma }%
				\left[ \omega \left( x-t\right) ^{\alpha }\right] \varphi \left( t\right)
				\, \mathrm dt\right| ^{p} \mathrm dx<\infty .
			\end{equation}%
			This can be done by recalling the well-known integral inequality%
			\begin{equation}
				\left| \int_{a}^{x}f\left( t\right) \mathrm dt\right|^{p}\leq
				\int_{a}^{x}\left| f\left( t\right) \right|^{p} \mathrm dt, \qquad 0<p\leq 1,
			\end{equation}%
			and the uniform bound of the function $e_{\alpha ,\beta, \omega }^{\gamma }\left(
			t\right) $ (see Theorem 3 of \citet{tom}). We obtain%
			\begin{align}
				\left\Vert (\mathbf{E}_{\alpha ,\beta, \omega, a^+}^{\gamma}\varphi)
				\right\Vert _{p}^{p}
				& \leq \int_{a}^{b}\left(
				\int_{a}^{x}\left\vert e_{\alpha ,\beta, \omega }^{\gamma }\left( x-t
				\right) \right\vert ^{p}\left\vert \varphi \left( t\right) \right\vert
				^{p} \mathrm dt\right) \mathrm dx \\
				& \leq \left( \frac{\Gamma \left( \gamma -\frac{\beta -1}{\alpha }\right)
				\Gamma \left( \frac{\beta -1}{\alpha }\right) }{\pi \alpha \omega ^{\frac{%
				\beta -1}{\alpha }}\Gamma \left( \gamma \right) \left[ \cos \left( \frac{\pi
				\alpha }{2}\right) \right] ^{\gamma -\frac{\beta -1}{\alpha }}}\right)
				^{p}\int_{a}^{b}\left( \int_{a}^{b}\left\vert \varphi \left(
				t\right) \right\vert ^{p} \mathrm dt\right) \mathrm dx \notag \\
				& \leq \left( b-a\right) \left( \frac{\textup{Be} \left( \gamma -\frac{\beta -1}{%
				\alpha }, \frac{\beta -1}{\alpha }\right) }{\pi \alpha
				\omega ^{\frac{\beta -1}{\alpha }} \left[ \cos
				\left( \frac{\pi \alpha }{2}\right) \right] ^{\gamma -\frac{\beta -1}{\alpha 
				}}}\right) ^{p}\left\Vert \varphi \right\Vert _{p}^{p} \notag.
			\end{align}%
			This completes the proof.
		\end{proof}

		\begin{theorem}
			Let $\alpha \in \left( 0,1\right)$, $\gamma ,\omega >0,$ and $\alpha \gamma
			>\beta -1>0$. If $\varphi \in L^{p}\left( a,b\right)$, $p>1$, then the
			integral operator $\mathbf{E}_{\alpha ,\beta, \omega, a^+ }^{\gamma }$ is
			bounded in $L^{1}\left( a,b\right) $ and 
			\begin{align}
				\left\Vert (\mathbf{E}_{\alpha ,\beta, \omega, a^+ }^{\gamma }\varphi)
				\right\Vert _{1}\leq M\left[ \frac{\left( b-a\right) ^{q+1}}{q+1}\right]
				^{1/q}\left\Vert \varphi \right\Vert _{p},
			\end{align}
			where $1/p+1/q=1$.
		\end{theorem}

		\begin{proof}
			By Fubini's theorem and H\"older inequality, we obtain%
			\begin{align}
				\left\Vert (\mathbf{E}_{\alpha ,\beta, \omega, a^+ }^{\gamma }\varphi)
				\right\Vert_{1}
				& =\int_{a}^{b}\left\vert \int_{a}^{x}\left(
				x-t\right)^{\beta -1}E_{\alpha ,\beta }^{\gamma }\left[ \omega \left(
				x-t\right)^{\alpha }\right] \varphi \left( t\right) \mathrm dt\right\vert \mathrm dx \\
				&\leq \int_{a}^{b}\left\vert \varphi \left( t\right) \right\vert
				\left( \int_{t}^{b}\left\vert e_{\alpha ,\beta, \omega }^{\gamma }\left(
				x-t \right) \right\vert \mathrm dx\right) \mathrm dt \notag \\
				&\leq M\int_{a}^{b}\left\vert \varphi \left( t\right) \right\vert
				\left( \int_{t}^{b} \mathrm dx\right) \mathrm dt \notag \\
				&\leq M\left( \int_{a}^{b}\left\vert \varphi \left( t\right)
				\right\vert ^{p}\mathrm dt\right) ^{1/p}\left( \int_{a}^{b}\left( b-t\right)
				^{q}\mathrm dt\right) ^{1/q} \notag \\
				&=M\left[ \frac{\left( b-a\right) ^{q+1}}{q+1}\right] ^{1/q}\left\Vert
				\varphi \right\Vert _{p}, \notag
			\end{align}
			where the constant $M$ is given by \eqref{givenby}.
		\end{proof}
		
		We now give a result on the boundedness of the Hilfer--Prabhakar derivative \eqref{hilg} in the space $L^1$.

		\begin{theorem}
			For $\mu \in \left( 0,1\right)$, $\nu \in \left[ 0,1\right] $ and $f\in L^1%
			\left( a,b\right) $ the operator $\mathcal{D}_{\rho ,\omega ,a^+}^{\gamma,
			\mu ,\nu }$ is bounded in the space $L^1\left( a,b\right) $ and 
			\begin{equation}
				\left\Vert (\mathcal{D}_{\rho ,\omega ,a^+}^{\gamma ,\mu ,\nu }f)\right\Vert
				_{1}\leq M_{1}M_{2}\left\Vert f\right\Vert _{1},
			\end{equation}%
			where 
			\begin{align}
				\label{muno}
				M_{1} & =\left( b-a\right) ^{ \nu \left( 1-\mu \right)
				}\sum_{k=0}^{\infty }\frac{\left\vert \left( \gamma \left( \nu
				-1\right) \right) _{k}\right\vert }{\left\vert \Gamma \left( \rho k+\nu
				\left( 1-\mu \right) \right) \right\vert \left[ \rho
				k+\nu \left( 1-\mu \right) \right] }>0, \\
				\label{mdue}
				M_{2} & = \left( b-a\right) ^{ \mu \nu -\mu -\nu
				}\sum_{k=0}^{\infty }\frac{\left\vert \left( \gamma \left( \nu
				-1\right) \right) _{k}\right\vert }{\left\vert \Gamma \left( \rho k+\mu \nu
				-\mu -\nu \right) \right\vert \left[ \rho k+
				\mu \nu -\mu -\nu \right] }>0,
			\end{align}
			and $\gamma,\omega \in \mathbb{R}$, $\rho>0$.
		\end{theorem}
		
		\begin{proof}
			Using the estimate given in Theorem 4 of \citet{Kil}, we obtain%
			\begin{align}
				\left\Vert (\mathcal{D}_{\rho ,\omega ,a^+}^{\gamma ,\mu ,\nu }f)\left( t\right)
				\right\Vert _{1}
				& =\left\Vert \left(\text{\textbf{E}}_{\rho ,\nu \left( 1-\mu
				\right) ,\omega ,a^+}^{-\gamma \nu } \frac{\mathrm d}{\mathrm dt}\left( \text{\textbf{E}%
				}_{\rho ,(1-\nu )\left( 1-\mu \right) ,\omega ,a^+}^{-\gamma (1-\nu
				)}f\right) \right) \left( t\right) \right\Vert _{1} \\
				& \leq M_{1}\left\Vert \frac{\mathrm d}{\mathrm dt}\left( \text{\textbf{E}}_{\rho ,(1-\nu
				)\left( 1-\mu \right) ,\omega ,a^+}^{-\gamma (1-\nu )}f\right) \left(
				t\right) \right\Vert _{1} \notag\\
				& =M_{1}\left\Vert \left( \text{\textbf{E}}_{\rho,
				(1-\nu )\left( 1-\mu \right) -1,\omega ,a^+}^{-\gamma (1-\nu )}f\right)
				\left( t\right) \right\Vert _{1} \notag \\
				& \leq M_{1}M_{2}\left\Vert f\right\Vert _{1}, \notag
			\end{align}%
			where $M_1$ and $M_2$ are the constants defined by \eqref{muno} and \eqref{mdue}.
		\end{proof}

		We now proceed to the study of the boundedness property of the regularized Prabhakar derivative \eqref{capl}
		and of the regularized Hilfer--Prabhakar derivative \eqref{regn}.
		In particular we prove $L^1$ boundedness.
		
		\begin{theorem}
			If $f \in W^{m,1}\left( a,b\right)$, $\gamma,\omega \in \mathbb{R}$, $m = \lceil \mu \rceil$, $\rho > 0$, then
			regularized Prabhakar derivative is bounded in $L^{1}\left( a,b\right)$ and the following inequality holds true:
			\begin{align}
				\left\Vert ({}^C\mathbf{D}^{\gamma}_{\rho, \mu, \omega,
				a^+}f) \right\Vert
				_{1}\leq \widetilde{K}\left\Vert f^{\prime }\right\Vert _{1},
			\end{align}
			where
			\begin{equation}
				\widetilde{K}=\left( b-a\right) ^{m-\mu}\sum_{k=0}^{\infty }%
				\frac{\left\vert \left( -\gamma \right) _{k}\right\vert }{\left\vert \Gamma
				\left( \rho k+m-\mu \right) \right\vert \left[ \rho
				k+m-\mu \right] }\frac{\left\vert \omega \left( b-a\right) ^{
				m-\mu}\right\vert ^{k}}{k!}.
			\end{equation}
		\end{theorem}

		\begin{proof}
			Using the $L^{1}$ estimate for the Prabhakar integral operator (see \citet{Kil}), we obtain%
			\begin{equation}
				\left\Vert ({}^C\mathbf{D}^{\gamma}_{\rho, \mu, \omega,
				a^+}f) \right\Vert
				_{1}=\left\Vert \left( \mathbf{E}^{-\gamma}_{\rho, m-\mu, \omega, a^+}\frac{\textup{d}^m}{\textup{d}t^m}
				f\right)\right\Vert _{1}\leq \tilde{K}\left\Vert f^{(m)}\right\Vert _{1}.
			\end{equation}
		\end{proof}
		
		\begin{theorem}
			If $f\in W^{1,1}\left( a,b\right)$, $\gamma,\omega \in \mathbb{R}$, $\mu \in (0,1)$, $\rho>0$. then the regularized version of
			the Hilfer--Prabhakar derivative $^{C}\mathcal{D}_{\rho ,\omega ,a^+}^{\gamma ,\mu }$
			is bounded in $L^{1}\left( a,b\right)$ and the following inequality holds:
			\begin{equation}
				\left\Vert ({}^{C}\mathcal{D}_{\rho ,\omega ,a^+}^{\gamma ,\mu } f) \right\Vert
				_{1}\leq K\left\Vert f^{\prime }\right\Vert _{1},
			\end{equation}%
			where
			\begin{equation}
				K=\left( b-a\right) ^{1-\mu}\sum_{k=0}^{\infty }%
				\frac{\left\vert \left( -\gamma \right) _{k}\right\vert }{\left\vert \Gamma
				\left( \rho k+1-\mu \right) \right\vert \left[ \rho
				k+1-\mu \right] }\frac{\left\vert \omega \left( b-a\right) ^{
				1-\mu}\right\vert ^{k}}{k!}.
			\end{equation}
		\end{theorem}

		\begin{proof}
			As in the proof of the preceding theorem we use again the $L^{1}$ estimate for the Prabhakar integral operator
			\citep{Kil}, obtaining
			\begin{equation}
				\left\Vert ({}^{C}\mathcal{D}_{\rho ,\omega ,a^+}^{\gamma ,\mu } f) \right\Vert
				_{1}=\left\Vert \left(\text{\textbf{E}}_{\rho ,1-\mu ,\omega ,a^+}^{-\gamma }\frac{\mathrm d%
				}{\mathrm dt}f\right)\right\Vert _{1}\leq K\left\Vert f^{\prime }\right\Vert _{1}.
			\end{equation}
		\end{proof}
		
		The following Proposition \ref{dej} and Theorem \ref{dej2} present basic composition relations involving Prabhakar integrals
		and Hilfer--Prabhakar derivatives. Some specific cases are highlighted.
		
		\begin{proposition}
			\label{dej}
			The following relationship holds true for any Lebesgue integrable function $%
			\varphi \in L^1\left( a,b\right)$:%
			\begin{align}
				& (\mathcal{D}_{\rho ,\omega ,a^+}^{\gamma ,\mu ,\nu }
				(\mathbf{E}_{\rho,\lambda,\omega,a^+}^{\delta } f))
				= (\mathbf{E}_{\rho ,\lambda -\mu,\omega ,a^+}^{\delta -\gamma }f),
			\end{align}
			where $\gamma ,\delta ,\omega \in \mathbb{R}$, $\rho,
				\lambda >0$, $\mu \in \left( 0,1\right)$, $\nu \in \left[ 0,1\right]$,
				$\lambda >\mu +\nu -\mu \nu$.
			In particular,
			\begin{equation}
				(\mathcal{D}_{\rho ,\omega ,a^+}^{\gamma ,\mu ,\nu }(\mathbf{E}_{\rho
				,\lambda ,\omega ,a^+}^{\gamma }f))=(I_{a^+}^{\lambda -\mu }f).
			\end{equation}
		\end{proposition}

		\begin{proof}
			Using the semi-group property of the Prabhakar integral operator \citep{Kil}, we obtain%
			\begin{align}
				(\mathcal{D}_{\rho ,\omega ,a^+}^{\gamma ,\mu ,\nu }(\text{\textbf{E}}_{\rho
				,\lambda ,\omega ,a^+}^{\delta }f))\left( t\right)
				& = \left(\text{\textbf{E}}_{\rho
				,\nu \left( 1-\mu \right) ,\omega ,a^+}^{-\gamma \nu } \frac{\mathrm d}{\mathrm dt}%
				\left( \text{\textbf{E}}_{\rho ,(1-\nu )\left( 1-\mu \right) ,\omega
				,a^+}^{-\gamma (1-\nu )}(\text{\textbf{E}}_{\rho ,\lambda ,\omega ,a^+}^{\delta
				}f)\right) \right) \left( t\right) \\
				& =\left(\text{\textbf{E}}_{\rho ,\nu \left( 1-\mu \right) ,\omega ,a^+}^{-\gamma \nu
				} \frac{\mathrm d}{\mathrm dt}\left( \text{\textbf{E}}_{\rho ,(1-\nu )\left( 1-\mu
				\right) +\lambda ,\omega ,a^+}^{-\gamma (1-\nu )+\delta }f\right) \right)
				\left( t\right) \notag \\
				& =\left( \text{\textbf{E}}_{\rho ,\nu \left( 1-\mu \right) ,\omega
				,a^+}^{-\gamma \nu }\left( \text{\textbf{E}}_{\rho ,(1-\nu )\left( 1-\mu
				\right) +\lambda -1,\omega ,a^+}^{-\gamma (1-\nu )+\delta }f\right)\right) \left(
				t\right) \notag \\
				& =(\text{\textbf{E}}_{\rho ,\lambda -\mu ,\omega ,a^+}^{\delta -\gamma }f)\left(
				t\right). \notag
			\end{align}
		\end{proof}

		\begin{proposition}
			\label{dej2}
			The following composition relationship holds true for any Lebesgue integrable function $\varphi \in L^1\left( a,b\right)$:
			\begin{align}
				& (I_{a^+}^{\lambda }(\mathcal{D}_{\rho ,\omega ,a^+}^{\gamma ,\mu ,\nu }\varphi)) =%
				(\mathcal{D}_{\rho ,\omega ,a^+}^{\gamma ,\mu ,\nu }(I_{a^+}^{\lambda }\varphi)) =%
				(\mathbf{E}_{\rho ,\lambda -\mu ,\omega ,a^+}^{-\gamma }\varphi),
			\end{align}
			where $\gamma, \omega \in 
			\mathbb{R}$, $\rho ,\lambda >0$, $\mu \in \left( 0,1\right)$, $\nu \in \left[ 0,1\right]$, $\lambda >\mu +\nu -\mu \nu$.
		\end{proposition}

		\begin{proof}
			It is sufficient to prove the first relation. The proof of the second follows the same lines. We have%
			\begin{align}
				\left( \mathcal{D}_{\rho ,\omega ,a^+}^{\gamma ,\mu ,\nu }(I_{a^+}^{\lambda
				}\varphi) \right) \left( t\right)
				& =\left( \text{\textbf{E}}_{\rho ,\nu \left(
				1-\mu \right) ,\omega ,a^+}^{-\gamma \nu } \frac{\mathrm d}{\mathrm dt}\left( \text{%
				\textbf{E}}_{\rho ,(1-\nu )\left( 1-\mu \right) ,\omega ,a^+}^{-\gamma (1-\nu
				)}I_{a^+}^{\lambda }\varphi\right) \right) \left( t\right) \\
				& =\left( \text{\textbf{E}}_{\rho ,\nu \left( 1-\mu \right) ,\omega
				,a^+}^{-\gamma \nu } \frac{\mathrm d}{\mathrm dt}\left( \text{\textbf{E}}_{\rho ,(1-\nu
				)\left( 1-\mu \right) +\lambda ,\omega ,a^+}^{-\gamma (1-\nu )}\varphi\right) %
				\right) \left( t\right) \notag \\
				& =\left( \text{\textbf{E}}_{\rho ,\nu \left( 1-\mu \right) ,\omega
				,a^+}^{-\gamma \nu }\left( \text{\textbf{E}}_{\rho ,(1-\nu )\left( 1-\mu
				\right) +\lambda -1,\omega ,a^+}^{-\gamma (1-\nu )}\varphi\right) \right) \left(
				t\right) \notag \\
				& = (\text{\textbf{E}}_{\rho ,\lambda -\mu ,\omega ,a^+}^{-\gamma
				}\varphi)(t). \notag
			\end{align}%
			On the other hand,
			\begin{align}
				\left( I_{a^+}^{\lambda }(\mathcal{D}_{\rho ,\omega ,a^+}^{\gamma ,\mu ,\nu
				}\varphi) \right) \left( t\right) & = \left(I_{a^+}^{\lambda }\left( \text{\textbf{E}}%
				_{\rho ,\nu \left( 1-\mu \right) ,\omega ,a^+}^{-\gamma \nu } \frac{\mathrm d}{%
				\mathrm dt}\left( \text{\textbf{E}}_{\rho ,(1-\nu )\left( 1-\mu \right) ,\omega
				,a^+}^{-\gamma (1-\nu )}\varphi\right) \right) \right)\left( t\right) \\
				& =\left( \text{\textbf{E}}_{\rho ,\nu \left( 1-\mu \right) +\lambda ,\omega
				,a^+}^{-\gamma \nu } \frac{\mathrm d}{\mathrm dt}\left( \text{\textbf{E}}_{\rho ,(1-\nu
				)\left( 1-\mu \right) +\lambda ,\omega ,a^+}^{-\gamma (1-\nu )}\varphi\right) %
				\right) \left( t\right) \notag \\
				& =\left( \text{\textbf{E}}_{\rho ,\nu \left( 1-\mu \right) ,\omega
				,a^+}^{-\gamma \nu }\left( \text{\textbf{E}}_{\rho ,(1-\nu )\left( 1-\mu
				\right) +\lambda -1,\omega ,a^+}^{-\gamma (1-\nu )}\varphi\right) \right) \left(
				t\right) \notag \\
				& =(\text{\textbf{E}}_{\rho ,\lambda -\mu ,\omega ,a^+}^{-\gamma
				}\varphi)(t). \notag
			\end{align}
		\end{proof}	

		\begin{example}
			\label{ex1}
			As a didactic example we calculate the (non regularized) Hilfer--Prabhakar derivative
			of the power function $t^{p-1}$, $p>1$, with $a=0$. As in Definition \ref{ge}, we
			consider $\mu\in (0,1)$, $\nu \in[0,1]$, $\gamma, \omega \in \mathbb{R}$, $\rho > 0$. We obtain
			\begin{align}
				& \hspace{-.5cm} \left( \mathcal{D}_{\rho ,\omega ,0^+}^{\gamma ,\mu ,\nu } t^{p-1}\right)
				\left( x\right) \\
				= {} & \left( \mathbf{E}_{\rho ,\nu \left( 1-\mu \right) ,\omega
				,0^+}^{-\gamma \nu } \frac{\mathrm d}{\mathrm dt}\left( \mathbf{E}_{\rho ,(1-\nu
				)\left( 1-\mu \right) ,\omega ,0^+}^{-\gamma (1-\nu )}t^{p-1}\right) \right)
				\left( x\right) \notag \\
				= {} & \Gamma \left( p\right) \left( \mathbf{E}_{\rho ,\nu \left( 1-\mu \right)
				,\omega ,0^+}^{-\gamma \nu } \frac{\mathrm d}{\mathrm dt}\left( t^{\left( 1-\nu \right)
				\left( 1-\mu \right) +p-1}E_{\rho ,(1-\nu )\left( 1-\mu \right) +p}^{-\gamma
				(1-\nu )}\left( \omega t^{\rho }\right) \right) \right) \left( x\right) \notag \\
				= {} & \Gamma \left( p\right) \left( \mathbf{E}_{\rho ,\nu \left( 1-\mu \right)
				,\omega ,0^+}^{-\gamma \nu } t^{\left( 1-\nu \right) \left( 1-\mu
				\right) +p-2}E_{\rho ,(1-\nu )\left( 1-\mu \right) +p-1}^{-\gamma (1-\nu
				)}\left( \omega t^{\rho }\right) \right) \left( x\right) \notag \\
				= {} &\Gamma \left( p\right) \int_{0}^{x}\left( x-t\right) ^{\nu \left(
				1-\mu \right) -1}E_{\rho ,\nu \left( 1-\mu \right) }^{-\gamma \nu }\left(
				\omega \left( x-t\right) ^{\rho }\right) \notag \\
				& \times t^{\left( 1-\nu \right) \left(
				1-\mu \right) +p-2}E_{\rho ,(1-\nu )\left( 1-\mu \right) +p-1}^{-\gamma
				(1-\nu )}\left( \omega t^{\rho }\right) \mathrm dt \notag \\
				= {} & x^{p-\mu -1}E_{\rho ,p-\mu }^{-\gamma }\left( \omega x^{\rho }\right). \notag
			\end{align}
			In the last step we applied Proposition \ref{calcolo}.
		\end{example}

		\begin{example}
			Similarly to Example \ref{ex1}, we calculate the (non regularized) Hilfer--Prabhakar derivative of the function $e_{\rho
			,\beta,\omega }^{\gamma }\left( t\right)$. We set
			$\mu\in (0,1)$, $\nu \in[0,1]$, $\gamma, \omega \in \mathbb{R}$, $\rho > 0$, $\beta>1$. In this case we have
			\begin{align}
				& \left( \mathcal{D}_{\rho ,\omega ,0^+}^{\gamma ,\mu ,\nu }e_{\rho ,\beta, \omega
				}^{\gamma }\left( t\right) \right) \left( x\right) \\
				& =\left(
				\mathbf{E}_{\rho ,\nu \left( 1-\mu \right) ,\omega ,0^+}^{-\gamma \nu }
				\frac{\mathrm d}{\mathrm dt}\left( \mathbf{E}_{\rho ,(1-\nu )\left( 1-\mu \right)
				,\omega ,0^+}^{-\gamma (1-\nu )}e_{\rho ,\beta, \omega}^{\gamma }\left(
				t\right) \right) \right) \left( x\right) \notag \\
				& =\left( \mathbf{E}_{\rho ,\nu \left( 1-\mu \right) ,\omega ,0^+}^{-\gamma \nu
				} \frac{\mathrm d}{\mathrm dt}\left( t^{\left( 1-\nu \right) \left( 1-\mu \right)
				+\beta -1}E_{\rho ,(1-\nu )\left( 1-\mu \right) +\beta }^{\gamma \nu }\left(
				\omega t^{\rho }\right) \right) \right) \left( x\right) \notag \\
				& =\left( \mathbf{E}_{\rho ,\nu \left( 1-\mu \right) ,\omega ,0^+}^{-\gamma \nu
				} t^{\left( 1-\nu \right) \left( 1-\mu \right) +\beta -2}E_{\rho
				,(1-\nu )\left( 1-\mu \right) +\beta -1}^{\gamma \nu }\left( \omega t^{\rho
				}\right) \right) \left( x\right) \notag \\
				& =x^{\beta-\mu-1}E_{\rho ,\beta-\mu}^{0}\left( \omega x^{\rho }\right)
				= \frac{x^{\beta-\mu-1}}{\Gamma(\beta-\mu)}. \notag
			\end{align}
			As in the above example in the second-to-last step we applied Proposition \ref{calcolo}.
		\end{example}
		
	\section{Opial- and Hardy-type inequalities}
		
		Opial-type inequalities have a great interest in mathematics in general and in
		specific fields such as the theory of differential equations, the theory of probability, approximations, and many others.
		The classical Opial inequality was introduced by \citet{MR0112926} and reads as follows:
		
		\begin{proposition}
			Let $f(y) \in C^1(0,h)$, $f(0)=f(h)=0$, and $f(y)>0$, $y \in (0,h)$. It holds
			\begin{align}
				\int_0^h \left| f(y) f'(y) \right| \textup{d}y \le \frac{h}{4} \int_0^h \left( f'(y) \right)^2 \textup{d}y,
			\end{align}
			where $h/4$ is the best possible constant.
		\end{proposition}
		During the past years the classical Opial inequality has been generalized by many authors in many different directions.
		For a thorough account the reader is suggested to consult the monograph on Opial-type inequality by \citet{MR1340422}.
		See also the recent paper by \citet{farid}.
		For further Opial-type inequalities involving the classical fractional operators see \citet{MR2513750}.
		
		In this section we describe and prove several Opial-type inequalities involving as differential and integral operators those
		analyzed in the previous section.
			
		\begin{theorem}
			let $f\in L^{1}\left( 0,x\right) .$ If $\alpha \in \left( 0,1\right)$, $%
			\gamma ,\omega >0$, $\alpha \gamma >\beta -1>0$, $p,q>1$, $1/p+1/q=1$, then the following
			inequality holds true:
			\begin{equation}
				\int_{0}^{x}\left\vert \left( \mathbf{E}_{\alpha ,\beta ,\omega
				,0^+}^{\gamma }f\right) \left( s\right) \right\vert \left\vert f\left(
				s\right) \right\vert \textup{d}s\leq K\frac{x^{2/p}}{2}\left(
				\int_{0}^{x}\left\vert f\left( s\right) \right\vert ^{q}\textup{d}s\right)
				^{2/q},
			\end{equation}%
			where%
			\begin{equation*}
				K=\frac{\textup{Be} \left( \gamma -\frac{\beta -1}{\alpha }, 
				\frac{\beta -1}{\alpha }\right) }{\pi \alpha \omega ^{\frac{\beta -1}{\alpha 
				}} \left[ \cos \left( \frac{\pi \alpha }{2}%
				\right) \right] ^{\gamma -\frac{\beta -1}{\alpha }}}.
			\end{equation*}
			\begin{proof}
				Applying H\"older inequality to the Prabhakar integral operator, we get%
				\begin{align}
					\left\vert \left( \mathbf{E}_{\alpha ,\beta ,\omega ,0^+}^{\gamma }f\right)
					\left( s\right) \right\vert
					& \leq \left( \int_{0}^{s}\left\vert
					e_{\alpha ,\beta, \omega }^{\gamma }\left(u\right) \right\vert
					^{p}\textup{d}u\right) ^{1/p}\left( \int_{0}^{s}\left\vert f\left( u\right)
					\right\vert ^{q}\textup{d}u\right) ^{1/q} \\
					&\leq \frac{\Gamma \left( \gamma -\frac{\beta -1}{\alpha }\right) \Gamma
					\left( \frac{\beta -1}{\alpha }\right) }{\pi \alpha \omega ^{\frac{\beta -1}{%
					\alpha }}\Gamma \left( \gamma \right) \left[ \cos \left( \frac{\pi \alpha }{2%
					}\right) \right] ^{\gamma -\frac{\beta -1}{\alpha }}}s^{1/p}\left(
					\int_{0}^{s}\left\vert f\left( u\right) \right\vert ^{q}\textup{d}u\right)
					^{1/q}. \notag
				\end{align}%
				Let $z\left( s\right) =\int_{0}^{s}\left\vert f\left( u\right)
				\right\vert^q \textup{d}u.$ Then $z^{\prime }\left( s\right) =\left\vert f\left(
				s\right) \right\vert ^{q},$ i.e.\ $\left\vert f\left( s\right) \right\vert
				=\left( z^{\prime }\left( s\right) \right) ^{1/q}.$ Hence,%
				\begin{equation*}
					\left\vert \left( \mathbf{E}_{\alpha ,\beta ,\omega ,0^+}^{\gamma }f\right)
					\left( s\right) \right\vert \left\vert f\left( s\right) \right\vert \leq
					Ks^{1/p}\left( z\left( s\right) z^{\prime }\left( s\right) \right) ^{1/q}.
				\end{equation*}%
				Applying again H\"older inequality, we obtain%
				\begin{align}
					\int_{0}^{x}\left\vert \left( \mathbf{E}_{\alpha ,\beta ,\omega
					,0^+}^{\gamma }f\right) \left( s\right) \right\vert \left\vert f\left(
					s\right) \right\vert \textup{d}s
					& \leq K\left( \int_{0}^{x}\left(
					s^{1/p}\right) ^{p}\textup{d}s\right) ^{1/p}\left( \int_{0}^{x}\left( z\left(
					s\right) z^{\prime }\left( s\right) \right) \textup{d}s\right) ^{1/q} \\
					& = K\frac{x^{2/p}}{2^{1/p}}\frac{\left( z\left( x\right) \right) ^{2/q}}{%
					2^{1/q}}=K\frac{x^{2/p}}{2}\left( \int_{0}^{x}\left\vert f\left(
					s\right) \right\vert ^{q}\textup{d}s\right) ^{2/q}. \notag
				\end{align}
			\end{proof}
		\end{theorem}
		
		The following theorem concerns an Opial-type inequality involving at the same time both Hilfer and Riemann--Liouville derivatives.
			
		\begin{theorem}
			Let $\mu \in (0,1)$, $\nu \in \left( 0,1\right]$, $f\in L^{1}\left( 0,x\right)$, $x>0$.
			Let furthermore $(D^{\mu +\nu -\mu \nu }_{0^+}f)\in L^{\infty }\left( 0,x\right)$. Then, for $%
			0<p<1$, the following inequality holds:
			\begin{align}
				\int_{0}^{x}\left\vert \left( D_{0^+}^{\mu ,\nu }f\right) \left(
				t\right) \left( D_{0^+}^{\mu +\nu -\mu \nu }f\right) \left( t\right)
				\right\vert \textup{d}t
				\geq \Theta(x) \left(
				\int_{0}^{x}\left\vert \left( D_{0^+}^{\mu +\nu -\mu \nu }f\right)
				\left( t\right) \right\vert \textup{d}t\right)^{2/q}, \notag
			\end{align}%
			where $1/p+1/q=1$ and
			\begin{align*}
				\Theta(x) = \frac{2^{-1/q}}{\Gamma \left( \nu \left( 1-\mu \right)
				\right) \left( \left( \nu \left( 1-\mu \right) -1\right) p+1\right) ^{1/p}}
				\frac{x^{\left( \nu \left( 1-\mu \right) p-p+2\right) /p}}{\left(
				\nu \left( 1-\mu \right) p-p+2\right) ^{1/p}}.
			\end{align*}
		\end{theorem}
			
		\begin{proof}
			Recall that the Hilfer derivative can be written as
			\begin{align}
				\left( D_{0^+}^{\mu ,\nu }f\right) \left( t\right)
				&= \left( I_{0^+}^{\nu
				\left( 1-\mu \right) }\frac{\textup{d}}{\textup{d}x}\left( I_{0^+}^{\left( 1-\mu \right) \left(
				1-\nu \right) }f\right) \right) \left( t\right) =\left( I_{0^+}^{\nu \left(
				1-\mu \right) }(D_{0^+}^{\mu +\nu -\mu \nu }f)\right) \left( t\right)  \\
				&= \frac{1}{\Gamma \left( \nu \left( 1-\mu \right) \right) }%
				\int_{0}^{t}\left( t-\tau \right) ^{\nu \left( 1-\mu \right)
				-1}\left( D_{0^+}^{\mu +\nu -\mu \nu }f\right) \left( \tau \right) \textup{d}\tau. \notag
			\end{align}%
			Applying now the reverse H\"older inequality (remember that here $p \in (0,1)$), we obtain%
			\begin{align}
				& \left\vert \left( D_{0^+}^{\mu ,\nu }f\right) \left( t\right) \right\vert  \\
				&\geq \frac{1}{\Gamma \left( \nu \left( 1-\mu \right) \right) }\left(
				\int_{0}^{t}\left( t-\tau \right) ^{\left( \nu \left( 1-\mu \right)
				-1\right) p}\textup{d}\tau \right) ^{1/p}\left( \int_{0}^{t}\left\vert \left(
				D_{0^+}^{\mu +\nu -\mu \nu }f\right) \left( \tau \right) \right\vert
				^{q}\textup{d}\tau \right) ^{1/q} \notag\\
				&= \frac{1}{\Gamma \left( \nu \left( 1-\mu \right) \right) }\frac{t^{\left(
				\left( \nu \left( 1-\mu \right) -1\right) p+1\right) /p}}{\left( \left( \nu
				\left( 1-\mu \right) -1\right) p+1\right) ^{1/p}}
				\left(
				\int_{0}^{t}\left\vert \left( D_{0^+}^{\mu +\nu -\mu \nu }f\right)
				\left( \tau \right) \right\vert ^{q}\textup{d}\tau \right) ^{1/q}. \notag
			\end{align}%

			Let us now define $z\left( t\right) =\int_{0}^{t}\left( \Phi \left( s\right) \right)^{q}\textup{d}s$,
			where $\Phi \left( t\right) =\left\vert \left( D_{0^+}^{\mu +\nu -\mu \nu
			}f\right) \left( t\right) \right\vert$. Then $\Phi \left( t\right) =\left( z^{\prime }\left( t\right) \right) ^{1/q}$,
			i.e. $\left\vert \left( D_{0^+}^{\mu +\nu -\mu \nu }f\right) \left( t\right)
			\right\vert =\left( z^{\prime }\left( t\right) \right) ^{1/q}$.
			Hence,%
			\begin{align}
				\int_{0}^{x}&\left\vert \left( D_{0^+}^{\mu ,\nu }f\right) \left(
				t\right) \left( D_{0^+}^{\mu +\nu -\mu \nu }f\right) \left( t\right)
				\right\vert \textup{d}t \\
				\geq {} & \frac{1}{\Gamma \left( \nu \left( 1-\mu \right)
				\right) }\int_{0}^{x}\frac{t^{\left( \left( \nu \left( 1-\mu \right)
				-1\right) p+1\right) /p}}{\left( \left( \nu \left( 1-\mu \right) -1\right)
				p+1\right) ^{1/p}}\left( z\left( t\right) z^{\prime }\left( t\right) \right)
				^{1/q}\textup{d}t \notag \\
				\geq {} & \frac{1}{\Gamma \left( \nu \left( 1-\mu \right) \right) \left( \left(
				\nu \left( 1-\mu \right) -1\right) p+1\right) ^{1/p}} \notag \\
				&\times \left( \int_{0}^{x}\left( t^{\left( \left( \nu \left( 1-\mu
				\right) -1\right) p+1\right) /p}\right) ^{p}\textup{d}t\right) ^{1/p}\left(
				\int_{0}^{x}\left( z\left( t\right) z^{\prime }\left( t\right)
				\right) \textup{d}t\right) ^{1/q} \notag \\
				= {} & \frac{1}{\Gamma \left( \nu \left( 1-\mu \right) \right) \left( \left( \nu
				\left( 1-\mu \right) -1\right) p+1\right) ^{1/p}}
				\frac{x^{\left( \left( \nu \left( 1-\mu \right) -1\right)
				p+2\right) /p}}{^{\left( \left( \nu \left( 1-\mu \right) -1\right)
				p+2\right) ^{1/p}}}\frac{\left( z\left( x\right) \right) ^{2/q}}{2^{1/q}} \notag \\
				= {} & \frac{2^{-1/q}}{\Gamma \left( \nu \left( 1-\mu \right) \right) \left(
				\left( \nu \left( 1-\mu \right) -1\right) p+1\right) ^{1/p}} \notag \\
				& \times \frac{x^{\left( \left( \nu \left( 1-\mu \right) -1\right)
				p+2\right) /p}}{^{\left( \left( \nu \left( 1-\mu \right) -1\right)
				p+2\right) ^{1/p}}}\left( \int_{0}^{x}\left\vert \left( D_{0^+}^{\mu
				+\nu -\mu \nu }f\right) \left( t\right) \right\vert \textup{d}t\right) ^{2/q}, \notag
			\end{align}
			and this proves the claimed formula.
		\end{proof}
			
		\begin{corollary}
			Let $\mu \in \left( 0,1\right)$, $f\in L^{1}\left( 0,b\right)$, $f^{\prime
			}\in L\left( 0,b\right)$, $0<x \le b$. For $0<p<1$ the following inequality
			holds true:
			\begin{align}
				\int_{0}^{x}\left\vert \left( ^{C}D_{0^+}^{\mu }f\right) \left(
				t\right) \cdot f^{\prime }\left( t\right) \right\vert \textup{d}t\geq \Theta \left(
				x\right) \left( \int_{0}^{x}\left\vert f^{\prime }\left( t\right)
				\right\vert \textup{d}t\right) ^{2/q},
			\end{align}
			where%
			\begin{align*}
				\Theta \left( x\right) =\frac{2^{-q}}{\Gamma \left( 1-\mu \right) \left(
				\left( -\mu p+1\right) \left( -\mu p+2\right) \right) ^{1/p}}x^{\left( -\mu
				p+2\right) /p}
			\end{align*}
			and $\frac{1}{p}+\frac{1}{q}=1.$
		\end{corollary}			
			
		\begin{theorem}
			Let $f\in L^{1}\left( 0,x\right) ,$ $x>0$ , $f\in AC^{m}\left( 0,x\right) $
			and $(\mathbf{D}_{\rho ,\mu ,\omega ,0^+}^{\gamma }f)\in L^{\infty }\left(
			0,x\right) ,$ $0<\rho <1,$ $\omega >0,$ $\gamma <0$ , $f^{\left( k\right)
			}\left( 0^+\right) =0,$ $k=0,1,2,\dots,m-1$, and $-\rho \gamma >m-\mu -1>0$. If $\
			p,q>1,$ $1/p+1/q=1$ then the following inequality holds true:
			\begin{align}
				\int_{0}^{x}\left\vert f^{\left( m\right) }\left( t\right) (\mathbf{D}%
				_{\rho ,\mu ,\omega ,0^+}^{\gamma }f)\left( t\right) \right\vert \textup{d}t\leq \Omega
				\left( x\right) \left( \int_{0}^{x}\left\vert f^{\left( m\right)
				}\left( s\right) \right\vert ^{q}\textup{d}s\right) ^{2/q},
			\end{align}%
			where%
			\begin{align*}
				\Omega \left( x\right) =\frac{\Gamma \left( -\gamma -\frac{m-\mu -1}{\rho }%
				\right) \Gamma \left( \frac{m-\mu -1}{\rho }\right) }{\pi \rho \omega ^{%
				\frac{m-\mu -1}{\rho }}\Gamma \left( -\gamma \right) \left[ \cos \left( 
				\frac{\pi \rho }{2}\right) \right] ^{-\gamma -\frac{m-\mu -1}{\rho }}}\frac{%
				x^{2/p}}{2}.
			\end{align*}
		\end{theorem}
			
		\begin{proof}
			Using H\"older inequality and the uniform estimate of the function $%
			e_{\rho ,\mu ,\omega }^{-\gamma }$ (see Theorem 3 of \citet{tom}),
			\begin{align*}
				\left\vert \left( \mathbf{D}_{\rho ,\mu ,\omega ,0^+}^{\gamma }f\right)
				\left( t\right) \right\vert
				& =\left\vert (\mathbf{E}_{\rho ,m-\mu ,\omega
				,0^+}^{-\gamma }f^{\left( m\right) })\left( t\right) \right\vert  \\
				& \leq \left( \int_{0}^{t}\left\vert e_{\rho ,m-\mu ,\omega }^{-\gamma
				}\left( s\right) \right\vert ^{p}\textup{d}s\right) ^{1/p}\left(
				\int_{0}^{t}\left\vert f^{\left( m\right) }\left( s\right)
				\right\vert ^{q}\textup{d}s\right)^{1/q} \\
				& \leq Mt^{1/p}\left( \int_{0}^{t}\left\vert f^{\left( m\right)
				}\left( s\right) \right\vert ^{q}\textup{d}s\right) ^{1/q}, 
			\end{align*}
			where%
			\begin{align*}
				M=\frac{\Gamma \left( -\gamma -\frac{m-\mu -1}{\rho }\right) \Gamma \left( 
				\frac{m-\mu -1}{\rho }\right) }{\pi \rho \omega ^{\frac{m-\mu -1}{\rho }%
				}\Gamma \left( -\gamma \right) \left[ \cos \left( \frac{\pi \rho }{2}\right) %
				\right] ^{-\gamma -\frac{m-\mu -1}{\rho }}}>0.
			\end{align*}
			Let $z\left( t\right) =\int_{0}^{t}\left\vert f^{\left( m\right)
			}\left( s\right) \right\vert ^{q}\textup{d}s$. Then $z^{\prime }\left( t\right) =
			\left\vert f^{\left( m\right) }\left( t\right) \right\vert ^{q}$, i.e. $%
			\left\vert f^{\left( m\right) }\left( t\right) \right\vert =\left( z^{\prime
			}\left( t\right) \right) ^{1/q}$.
			Hence, 
			\begin{align*}
				\int_{0}^{x}\left\vert f^{\left( m\right) }\left( t\right) (\mathbf{D}%
				_{\rho ,\mu ,\omega ,0^+}^{\gamma }f)\left( t\right) \right\vert \textup{d}t &\leq
				M\int_{0}^{x}t^{1/p}\left( z\left( t\right) z^{\prime }\left(
				t\right) \right) ^{1/q}\textup{d}t \\
				&\leq M\left( \int_{0}^{x}\left( t^{1/p}\right) ^{p}\textup{d}t\right)
				^{1/p}\left( \int_{0}^{x}\left( z\left( t\right) z^{\prime }\left(
				t\right) \right) \textup{d}t\right) ^{1/q} \\
				&= M\frac{x^{2/p}}{2^{1/p}}\frac{\left( z\left( x\right) \right) ^{2/q}}{%
				2^{1/q}}=M\frac{x^{2/p}}{2}\left( \int_{0}^{x}\left\vert f^{\left(
				m\right) }\left( t\right) \right\vert ^{q}\textup{d}t\right) ^{2/q}.
			\end{align*}
		\end{proof}

		The following two results are given without proof. However, the claimed formulae can be derived with similar methods to those
		implemented in the proofs of the previous theorems of this section.			
			
		\begin{corollary}
			Let $\ f\in L^{1}\left( 0,x\right)$, $x>0$. Let furthermore $f\in AC^{m}\left( 0,x\right)$, $m\in 
			\mathbb{N}$, $({}^{C}\mathbf{D}_{\rho ,\mu ,\omega ,0^+}^{\gamma}f)\in L^{\infty}\left( 0,x\right)$,
			$0<\rho <1$, $\omega >0$, $\gamma <0$ and $-\rho \gamma >m-\mu -1>0$. If $p,q>1$,
			$1/p+1/q=1$, then the following inequality holds true:
			\begin{align}
				\int_{0}^{x}\left\vert f^{\left( m\right) }\left( t\right)
				({}^{C}\mathbf{D}_{\rho ,\mu ,\omega ,0^+}^{\gamma }f)\left( t\right) \right\vert \textup{d}t\leq
				\Omega \left( x\right) \left( \int_{0}^{x}\left\vert f^{\left(
				m\right) }\left( s\right) \right\vert ^{q}\textup{d}s\right) ^{2/q},
			\end{align}%
			where%
			\begin{align*}
				\Omega \left( x\right) =\frac{\Gamma \left( -\gamma -\frac{m-\mu -1}{\rho }%
				\right) \Gamma \left( \frac{m-\mu -1}{\rho }\right) }{\pi \rho \omega ^{%
				\frac{m-\mu -1}{\rho }}\Gamma \left( -\gamma \right) \left[ \cos \left( 
				\frac{\pi \rho }{2}\right) \right] ^{-\gamma -\frac{m-\mu -1}{\rho }}}\frac{%
				x^{2/p}}{2}. 
			\end{align*}
		\end{corollary}
			
		\begin{theorem}
			Let $\ f\in L^{1}\left( 0,x\right)$, $x>0$, $f\ast e_{\rho ,\left( 1-\nu
			\right) \left( 1-\mu \right), \omega }^{-\gamma(1-\nu)}\in AC^{1}\left( 0,b\right)$.
			Let furthermore $(\mathbf{D}_{\rho ,\omega ,0^+}^{\gamma ,\mu ,\nu }f)\in L^{\infty }\left(
			0,x\right)$, $\mu \in \left( 0,1\right)$, $\nu \in \left[ 0,1\right]
			,0<\rho <1$, $\omega >0$, $\gamma <0$, and $-\rho \gamma >1-\mu >0.$ If
			$p,q>1$, $1/p+1/q=1$, then the following inequality holds true:
			\begin{align}
				& \int_{0}^{x}\left\vert \left( \mathbf{D}_{\rho ,\omega ,0^+}^{\gamma
				,\mu ,\nu }f\right) \left( t\right) \cdot \frac{\textup{d}}{\textup{d}t}\left( \mathbf{E}%
				_{\rho ,(1-\nu )\left( 1-\mu \right) ,\omega ,0^+}^{-\gamma (1-\nu )}f\right)
				\left( t\right) \right\vert \textup{d}t \\
				& \leq \widetilde{\Omega }\left( x\right) \left(
				\int_{0}^{x}\left\vert \frac{\textup{d}}{\textup{d}t}\left( \mathbf{E}_{\rho ,(1-\nu
				)\left( 1-\mu \right) ,\omega ,0^+}^{-\gamma (1-\nu )}f\right) \left(
				t\right) \right\vert ^{q}\textup{d}t\right) ^{2/q},  \notag
			\end{align}%
			where%
			\begin{align*}
				\widetilde{\Omega }\left( x\right) =\frac{\Gamma \left( -\gamma \nu -\frac{%
				\nu (1-\mu )}{\rho }\right) \Gamma \left( \frac{\nu (1-\mu )}{\rho }\right) 
				}{\pi \rho \omega ^{\frac{\nu (1-\mu )}{\rho }}\Gamma \left( -\gamma \nu
				\right) \left[ \cos \left( \frac{\pi \rho }{2}\right) \right] ^{-\gamma \nu -%
				\frac{\nu (1-\mu )}{\rho }}}\frac{x^{2/p}}{2}>0.
			\end{align*}
		\end{theorem}
			
		\begin{theorem}[Hardy-type inequality]
			Let $p,q>1$, $1/p+1/q=1$,
			$\alpha,\beta ,\gamma ,\omega >0$.
			If $f\in L^{q}\left( a,b\right)$, $a < b$,
			then the following inequality holds true:
			\begin{align}
				\label{aa}
				\int_{a}^{b}\left\vert \left( \mathbf{E}_{\alpha ,\beta ,\omega
				,a^+}^{\gamma }f\right) \left( t\right) \right\vert ^{q}\textup{d}t\leq
				C\int_{a}^{b}\left\vert f\left( t\right) \right\vert ^{q}\textup{d}	t,
			\end{align}
			where $C = \left[ e_{\alpha ,\beta +2,\omega }^{\gamma }\left( b-a\right) \right]^{q}$.
			If $\alpha \in \left( 0,1\right)$, $\alpha \gamma >\beta -1$, we have				
			\begin{align}
				\label{bb}
				\int_{a}^{b}\left\vert \left( \mathbf{E}_{\alpha ,\beta ,\omega
				,a^+}^{\gamma }f\right) \left( t\right) \right\vert ^{q}\textup{d}t\leq
				K\int_{a}^{b}\left\vert f\left( t\right) \right\vert ^{q}\textup{d}	t,
			\end{align}
			where $K = M\left( b-a\right) ^{q/p+1}$. The constant $M$ is given by the right hand side of \eqref{poga}.

			\begin{proof}
				We prove only inequality \eqref{aa} as
				the proof of \eqref{bb} follows the same lines as the proof of
				Theorem \ref{primoteo}.
				By applying H\"older inequality we have
				\begin{align}
					\left\vert \left( \mathbf{E}_{\alpha ,\beta ,\omega ,a^+}^{\gamma }f\right)
					\left( t\right) \right\vert
					\leq & \int_{a}^{t}\left\vert e_{\alpha
					,\beta ,\omega }^{\gamma }\left( t-\tau \right) \right\vert \left\vert
					f\left( \tau \right) \right\vert \textup{d}\tau \\
					\leq & \left( \int_{a}^{t}\left\vert e_{\alpha ,\beta ,\omega
					}^{\gamma }\left( t-\tau \right) \right\vert ^{p}\textup{d}\tau \right) ^{1/p}\left(
					\int_{a}^{t}\left\vert f\left( \tau \right) \right\vert ^{q}\textup{d}\tau
					\right) ^{1/q} \notag \\
					\leq & e_{\alpha ,\beta+1 ,\omega }^{\gamma }\left( t-a\right)
					\left( \int_{a}^{b}\left\vert f\left(
					t\right) \right\vert ^{q}\textup{d}t\right) ^{1/q} \notag.
				\end{align}%
				Thus we have%
				\begin{align}
					\left\vert \left( \mathbf{E}_{\alpha ,\beta ,\omega ,a^+}^{\gamma }f\right)
					\left( t\right) \right\vert ^{q}\leq \left[ e_{\alpha ,\beta+1 ,\omega
					}^{\gamma }\left( t-a\right) \right] ^{q}\left(
					\int_{a}^{b}\left\vert f\left( t\right) \right\vert ^{q}\textup{d}t\right), 
				\end{align}
				for every $t\in \left[ a,b\right]$. Consequently we obtain
				\begin{align}
					\int_{a}^{b}\left\vert \left( \mathbf{E}_{\alpha ,\beta ,\omega
					,a^+}^{\gamma }f\right) \left( t\right) \right\vert ^{q}\textup{d}t
					& \leq \int_a^b \left[
					e_{\alpha ,\beta+1 ,\omega }^{\gamma }\left( t-a\right) \right]^{q} \textup{d} t
					\left( \int_{a}^{b}\left\vert f\left( t\right)
					\right\vert ^{q}\textup{d}t\right) \\
					& \leq \left( \int_a^b 
					e_{\alpha ,\beta+1 ,\omega }^{\gamma }\left( t-a\right) \textup{d} t \right)^q
					\left( \int_{a}^{b}\left\vert f\left( t\right)
					\right\vert ^{q}\textup{d}t\right) \notag \\
					& = \left[ e_{\alpha ,\beta+2 ,\omega }^{\gamma }\left( b-a\right) \right]^{q}
					\int_{a}^{b}\left\vert f\left( t\right)
					\right\vert ^{q}\textup{d}t \notag.
				\end{align}
				In the last step we made use of formula (5.5.19), page 100, in \citet{MR3244285}.
			\end{proof}
		\end{theorem}

	\section{Applications to probability theory}
		
		In this section we make some remarks about connections of the described operators with the theory of probability
		and of stochastic processes.
		
		First we should recall that some applications to probability have been already discussed in several articles. As an example
		we mention the paper by \citet{d2013fractional} in which the regularized Prabhakar derivative has been defined with the purpose of
		introducing a broad class of stochastic processes related to some partial differential equations of parabolic and of hyperbolic type.
		The paper by \citet{MR3239686} instead, presents an analysis of a generalized Poisson process in which the governing
		difference-differential equations contain a regularized Prabhakar derivative. The interested reader is encouraged to
		consult also the references therein.
		
		Consider the Wright function $\phi(-\alpha,\rho;z)$, $\alpha \in (0,1)$, $\rho,z \in \mathbb{R}$,
		defined as the convergent series
		\begin{align*}
			\phi(-\alpha,\rho;z) = \sum_{r=0}^\infty \frac{z^r}{r!\Gamma(-\alpha r+\rho)}.
		\end{align*}
		This can be used to construct
		a probability density function which proves to be interesting in many aspects. We will refer basically to the papers by
		\citet{MR0280762}, by \citet{MR1752379}, and by \citet{tom}, for the properties of Wright functions.
			
		We know \citep[see also][]{MR2218073,bateman1955higher} that the Mellin transform of $\phi$
		for the following specific choice of the parameters reads
		\begin{align}
			\mathcal{M}\left(\phi(-\alpha,\beta-\alpha \gamma;-z)\right)(\gamma) =\frac{\Gamma(\gamma)}{\Gamma(\beta)},
			\qquad \alpha,\beta \in (0,1), \: \gamma > 0, \: \beta \ge \alpha \gamma.
		\end{align}
		From this, it is clear that we can consider a random variable, say $X$, supported on $\mathbb{R}^+$, such that
		its probability density function is
		\begin{align}
			\label{ffff}
			g(x) = \frac{\Gamma(\beta)}{\Gamma(\gamma)} x^{\gamma-1} \phi(-\alpha,\beta-\gamma \alpha; -x) \mathds{1}_{(0,\infty)}(x).
		\end{align}
		and such that $\mathbb{E}X = \gamma \Gamma(\beta)/\Gamma(\alpha+\beta)$. The above definition is justified by the
		positivity of \eqref{ffff} as remarked in the proof of Theorem 2 in \citet{tom}.
			
		The above random variable is particularly interesting in that it generalizes the marginal law of an inverse stable subordinator.
		Let us thus consider an $\alpha$-stable subordinator $V^\alpha(t)$, $t \ge 0$, that is an increasing spectrally positive L\'evy
		process such that $\mathbb{E} \exp (-\lambda V^\alpha(t)) = \exp (-t\lambda^\alpha)$ \citep{MR2250061,MR1406564}. Let us call
		$E^\alpha(t) = \inf\{ s>0 \colon V^\alpha(s) \notin (0,t) \}$. The marginal probability density function
		$f(x,t) = \mathbb{P}(E^\alpha(t) \in \textup{d}x) / \textup{d}x$ satisfies the fractional pde
		\begin{align}
			({}^C D^\alpha_{0^+,t} f)(x,t) = -\frac{\textup{d}}{\textup{d}x} f(x,t), \qquad x>0, \: t \ge 0, 
		\end{align} 
		and can be explicitly written as
		\begin{align}
			\label{kumn}
			f(x,t) = t^{-\alpha} \phi\left( -\alpha,1-\alpha; - \frac{x}{t^\alpha} \right), \qquad x>0, \: t \ge 0.
		\end{align}
		The function \eqref{ffff} can be rewritten by considering a fixed time $t$ as
		\begin{align}
			\label{gggg}
			g(x,t) = \frac{\Gamma(\beta)}{\Gamma(\gamma)} t^{-\gamma \alpha} x^{\gamma-1} \phi \left( -\alpha,
			\beta-\alpha \gamma; -\frac{x}{t^\alpha} \right), \qquad x>0, \: t \ge 0,
		\end{align}
		which clearly generalizes \eqref{ffff} and \eqref{kumn}.
			
		\begin{theorem}
			The space-Laplace transform of the probability density function \eqref{gggg} writes
			\begin{align}
				\tilde{g} (s,t) = \int_0^\infty e^{-sx} \frac{\Gamma(\beta)}{\Gamma(\gamma)} t^{-\gamma \alpha} x^{\gamma-1}
				\phi \left(-\alpha,\beta-\alpha\gamma; -\frac{x}{t^\alpha} \right) \textup{d}x
				= \Gamma(\beta) E_{\alpha,\beta}^\gamma(-st^\alpha).
			\end{align}
			Moreover its space-time-Laplace transform reads
			\begin{align}
				\tilde{\tilde{g}}(s,\varpi) = \Gamma(\beta) \frac{\varpi^{\alpha \gamma-\beta}}{(\varpi^\alpha + s)^\gamma}, \qquad s>0,
				\: \varpi>0.
			\end{align}
		\end{theorem}

		\begin{proof}
			By using the contour integral representation on the Hankel path\footnote{The Hankel path starts at
			$(-\infty,-\varepsilon)$, $\varepsilon > 0$, proceeds to the origin on the lower half-plane, circles the origin counterclockwise
			and then returns to $(-\infty,\varepsilon)$ along the upper half-plane.} of the Wright function $\phi$ we obtain
			\begin{align}
				\tilde{g} (s,t) & = \int_0^\infty e^{-sx} \frac{\Gamma(\beta)}{\Gamma(\gamma)} t^{-\gamma \alpha} x^{\gamma-1}
				\frac{1}{2 \pi i} \int_{\textup{Ha}} e^{\zeta -\frac{x}{t^\alpha}\zeta^\alpha} \zeta^{\alpha\gamma-\beta}\, \textup{d}\zeta
				\, \textup{d}x \\
				& = \frac{\Gamma(\beta)}{\Gamma(\gamma)} \frac{t^{-\gamma\alpha}}{2 \pi i} \int_{\textup{Ha}}
				e^{\zeta} \zeta^{\alpha \gamma-\beta} \textup{d}\zeta \int_0^\infty e^{-x\left( s+\frac{\zeta^\alpha}{t^\alpha} \right)}
				x^{\gamma-1} \textup{d}x \notag \\
				& = \frac{\Gamma(\beta)t^{-\gamma \alpha}}{2 \pi i} \int_{\textup{Ha}} \frac{e^\zeta \zeta^{\alpha \gamma-\beta}}{
				\left( s+\frac{\zeta^\alpha}{t^\alpha} \right)^\gamma} \textup{d}\zeta \notag \\
				& = \frac{\Gamma(\beta)}{2\pi i} \int_{\textup{Ha}} \frac{e^\zeta \zeta^{\alpha \gamma-\beta}}{
				\left( t^\alpha s+\zeta^\alpha \right)^\gamma} \textup{d}\zeta \notag \\
				& = \Gamma(\beta) E_{\alpha,\beta}^\gamma(-st^\alpha). \notag
			\end{align}
			The last step is justified by using the contour integral representation of the reciprocal of the Gamma function,
			\begin{align}
				\frac{1}{\Gamma(\eta)} = \int_{\textup{Ha}}
				e^\zeta \zeta^{-\eta} \textup{d}\zeta,
			\end{align}
			and by the following calculation:
			\begin{align}
				E_{\alpha,\beta}^\gamma(z) & = \sum_{r=0}^\infty \frac{z^r \Gamma(\gamma+r)}{\Gamma(\gamma)r!\Gamma(\alpha r+\beta)} \\
				& = \frac{1}{2\pi i} \sum_{r=0}^\infty \frac{z^r\Gamma(\gamma+r)}{\Gamma(\gamma)r!} \int_{\textup{Ha}}
				e^\zeta \zeta^{-\alpha r-\beta} \textup{d}\zeta \notag \\
				& =\frac{1}{2 \pi i} \int_{\textup{Ha}} e^{\zeta} \zeta^{-\beta} \textup{d}\zeta \sum_{r=0}^\infty
				\binom{\gamma+r-1}{r} (z\zeta^{-\alpha})^r \notag \\
				& =\frac{1}{2 \pi i} \int_{\textup{Ha}} e^{\zeta} \zeta^{-\beta} \textup{d}\zeta \sum_{r=0}^\infty
				\binom{-\gamma}{r} (-z\zeta^{-\alpha})^r \notag \\
				& = \frac{1}{2 \pi i} \int_{\textup{Ha}} e^\zeta \zeta^{-\beta} (1-z\zeta^{-\alpha})^{-\gamma} \textup{d}\zeta \notag \\
				& = \frac{1}{2 \pi i} \int_{\textup{Ha}} \frac{e^\zeta \zeta^{-\beta+\alpha \gamma}}{(\zeta^\alpha-z)^\gamma}
				\textup{d}\zeta. \notag
			\end{align}
			The space-time-Laplace transform can now be easily obtained as
			\begin{align}
				\tilde{\tilde{g}}(s,\varpi) = \Gamma(\beta) \frac{\varpi^{\alpha \gamma-\beta}}{(\varpi^\alpha + s)^\gamma},
				\qquad s>0, \: \varpi>0.
			\end{align}
		\end{proof}
			
		\begin{theorem}
			The function $\tilde{g}(s,t)$
			satisfies the fractional equation
			\begin{align}
				\label{equequ}
				\Bigl(\bigl(D^\alpha_{0^+,t} +s \bigr)^\gamma \tilde{g}\Bigr)
				(s,t) = \Gamma(\beta) \frac{t^{\beta-\alpha\gamma-1}}{\Gamma(\beta-\alpha\gamma)},
				\qquad s>0, \: t \ge 0.
			\end{align} 
		\end{theorem}
			
		\begin{proof}
			The proof follows from the properties of the time-Laplace transform of $\tilde{g}$.
		\end{proof}
			
		\begin{remark}
			Notice that, for $\beta=1$, the function $\tilde{g}$ and the governing equation \eqref{equequ} are connected
			to the Havriliak--Negami relaxation (see \citet{stanislavsky2010subordination}
			for more detailed information).
		\end{remark}
			
		Let us now recall Theorem 1 of \citet{tom}	

		\begin{theorem*}[Theorem 1, \citet{tom}]
			For all $\alpha \in (0,1]$, $\beta>0$, $\gamma>0$, $t >0$, we have
			\begin{align}
				\label{eq1}
				e^\gamma_{\alpha,\beta,1}(t) = \frac{1}{2\pi i} \int_{\text{Br}_{\sigma_0}}
				e^{st} \frac{s^{\alpha \gamma-\beta}}{(s^\alpha+1)^\gamma}
				\textup{d} s = \mathcal{L}_t (K_{\alpha,\beta}^\gamma),
			\end{align}
			where $\mathcal{L}_t$ stands for the Laplace transform with $\text{Br}_{\sigma_0}$
			the Bromwich path (i.e.\ $\{ s=\sigma + i \tau \colon \sigma \ge \sigma_0, \: \tau \in \mathbb{R} \}$) and
			\begin{align}
				K_{\alpha,\beta}^\gamma(r) = \frac{r^{\alpha\gamma-\beta}}{\pi} \frac{\sin \left( \gamma \arctan
				\left( \frac{r^\alpha \sin(\pi \alpha)}{r^\alpha \cos (\pi \alpha)+1} \right) +\pi (\beta -\alpha \gamma) \right)}{
				\left( r^{2\alpha} + 2 r^\alpha \cos(\pi \alpha)+1 \right)^{\gamma/2}} .
			\end{align}
			Furthermore, for all $\alpha \in (1,2]$, $\beta>0$, and $\gamma=n \in \mathbb{N}$,
			\begin{align}
				\label{eq2}
				e^n_{\alpha,\beta,1}(t) = {} & \mathcal{L}_t^{-1}\left(\frac{s^{\alpha n-\beta}}{(s^\alpha+1)^n}\right)
				+ \frac{2(-1)^{n-1}}{\alpha^n(n-1)!}
				e^{t \cos (\pi/\alpha)} \\
				& \times \cos \left( t \sin (\pi/\alpha)- \frac{\pi (\beta-1)}{\alpha} \right)
				\sum_{l=0}^{n-1}\frac{(1-n)_l c_l}{(\alpha n -\beta -n+2)_l}, \notag
			\end{align}
			where
			\begin{align}
				\mathcal{L}_t^{-1}\left(\frac{s^{\alpha n-\beta}}{(s^\alpha+1)^n}\right)
				= \frac{1}{2\pi i} \int_{\text{Br}_{\sigma_0}}
				e^{st} \frac{s^{\alpha n-\beta}}{(s^\alpha+1)^n}
				\textup{d} s = \mathcal{L}_t (K_{\alpha,\beta}^n),
			\end{align}
			and where $c_l$ is given by $c_l=(-1)^{l} \sum_{\substack{j_1+\dots +j_n=l \\ 0\le j_1\le \dots\le j_n\le l}}
			b_{j_1}^* \cdot \ldots \cdot b_{j_n}^*$, with
			\begin{align*}
				b_j^*= \delta_{0,j} + q^{-j}(1-\delta_{0,j})\left|
				\begin{array}{cccccc}
					\binom{q}{2} & q & 0 & 0 & \dots & 0 \\
					\binom{q}{3} & \binom{q}{2} & q & 0 & \dots & 0 \\
					\vdots & \vdots & \vdots & \vdots & \dots & \vdots \\
					\binom{q}{j} & \binom{q}{j-1} & \binom{q}{j-2} & \dots & \dots & q \\
					\binom{q}{j+1} & \binom{q}{j} & \binom{q}{j-1} & \dots & \dots & \binom{q}{2}
				\end{array}			
				\right|, \qquad j \in \mathbb{N} \cup \{0\}, \: q > 0.
			\end{align*}
		\end{theorem*}
		
		The function $e^\gamma_{\alpha,\beta,1}(t)$ is completely monotone whenever $\alpha \in (0,1]$, $0<\alpha \gamma \le \beta \le 1$
		\citep{cap,tom}, and therefore by the Bernstein Theorem \citep{schi} the spectral function $K_{\alpha,\beta}^\gamma(r)$ is non-negative
		for the same range of the parameters.
			
		Moreover, from the above formulae \eqref{eq1} and \eqref{eq2}, we also derive the following result.

		\begin{theorem}	
			We have that	
			\begin{align}
				\int_0^\infty K_{\alpha,1}^\gamma(r) \, \textup{d}r =
				\begin{cases}
					1, & \alpha \in (0,1],\: \gamma > 0, \\
					1-\frac{2\left( -1\right)^{n-1}}{\alpha^{n}\left( n-1\right) !}
					\sum_{l=0}^{n-1}\frac{\left( 1-n\right) _{l}c_{l}}{\left( n\left(
					\alpha -1\right) +1\right) _{l}}, &\alpha \in \left( 1,2\right], \: \gamma =n\in \mathbb{N}.
				\end{cases}
			\end{align} 
		\end{theorem}
		
		\begin{proof}
			We let $t \to 0^+$ in the formulae \eqref{eq1} and \eqref{eq2} obtaining
			\begin{align}
				\label{tibor}
				1 = \lim_{t\to 0^+} e^{\gamma}_{\alpha,1,1}(t) = \int_0^\infty K_{\alpha,1}^\gamma(r) \, \textup{d}r,
				\qquad \alpha \in (0,1],\: \gamma > 0,
			\end{align}
			and
			\begin{align}
				\label{tibor2}
				1 = \int_0^\infty K_{\alpha,1}^\gamma(r) \, \textup{d}r +
				\frac{2\left( -1\right)^{n-1}}{\alpha^{n}\left( n-1\right) !}
				\sum_{l=0}^{n-1}\frac{\left( 1-n\right) _{l}c_{l}}{\left( n\left(
				\alpha -1\right) +1\right) _{l}}, \qquad \alpha \in \left( 1,2\right],
				\: \gamma =n\in \mathbb{N}.
			\end{align}
			From this, the claim easily follows.
		\end{proof}

		\begin{corollary}
			If $\gamma=1$, formulae \eqref{tibor} and \eqref{tibor2}, reduce to
			\begin{align}
				\int_0^\infty K_{\alpha}(r) \textup{d}r =
				\begin{cases}
					1, & \alpha \in (0,1], \\
					1-\frac{2}{\alpha}, & \alpha \in (1,2].
				\end{cases}
			\end{align}
		\end{corollary}
		The kernel $K_\alpha(r)$ has been thoroughly studied in the literature (for more in-depth information see e.g.\
		\citet{gorgor2} and the references therein) while the general spectral function $K_{\alpha,\beta}^\gamma(r)$ has been
		recently extensively
		analyzed in \citet{garrappa}.

		\begin{figure}
			\centering
			\includegraphics[scale=.42]{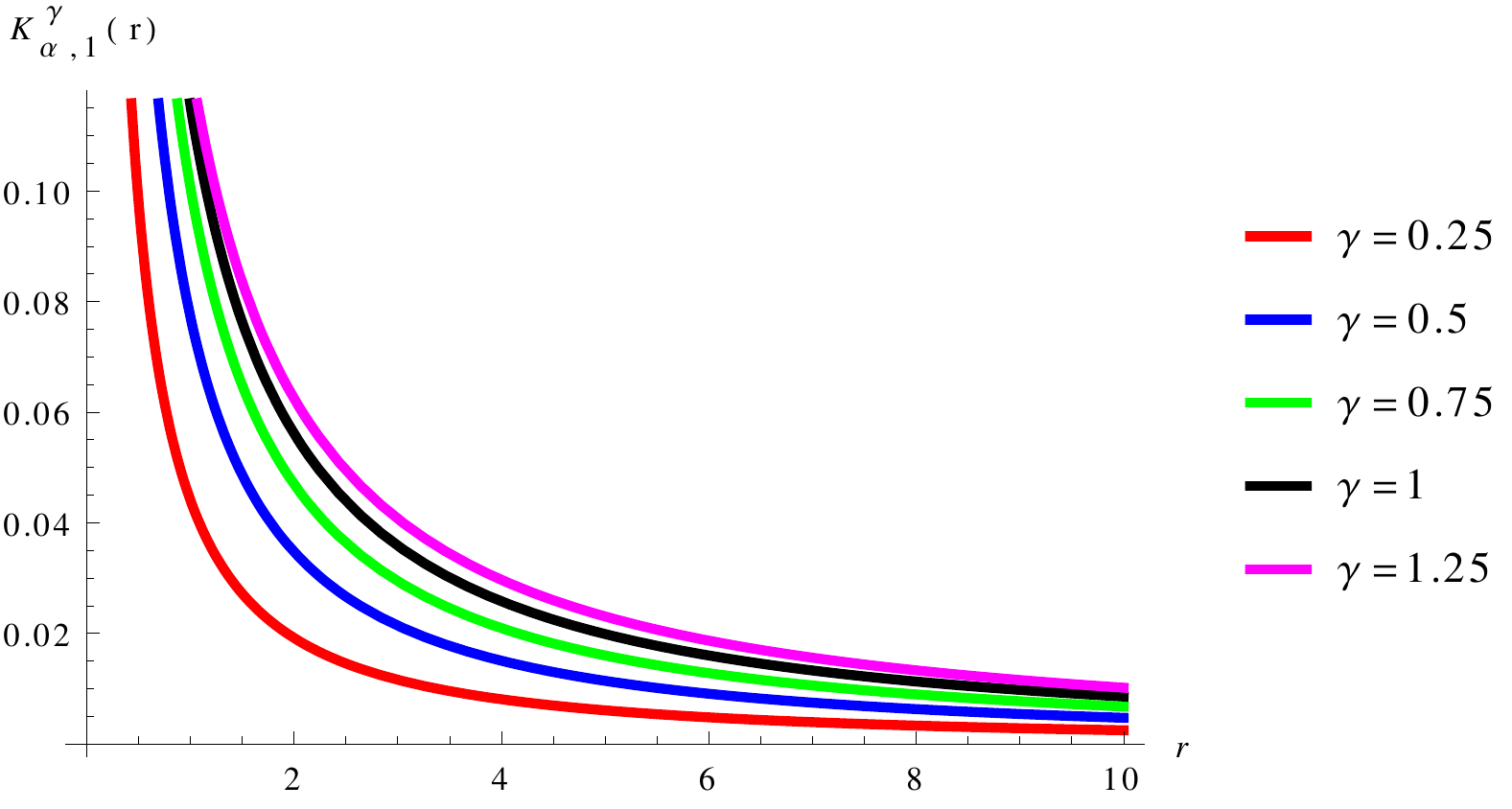}
			\includegraphics[scale=.42]{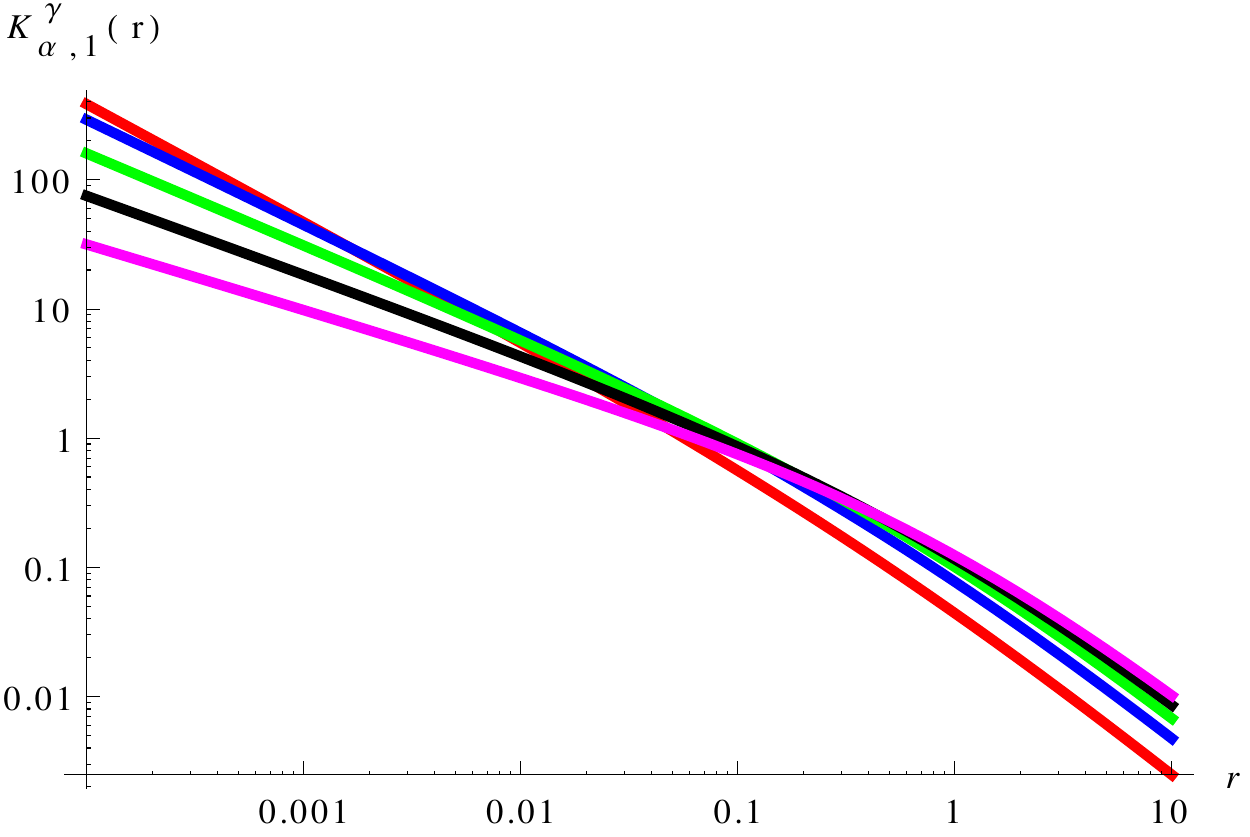}
			\caption{\label{bababa}Plot (left) and log-log-plot (right)
			of the function \eqref{bababa0} for $\alpha=0.4$ and $\gamma=(0.25,0.5,0.75,1,1.25)$.}
		\end{figure}

		\begin{figure}
			\centering
			\includegraphics[scale=.42]{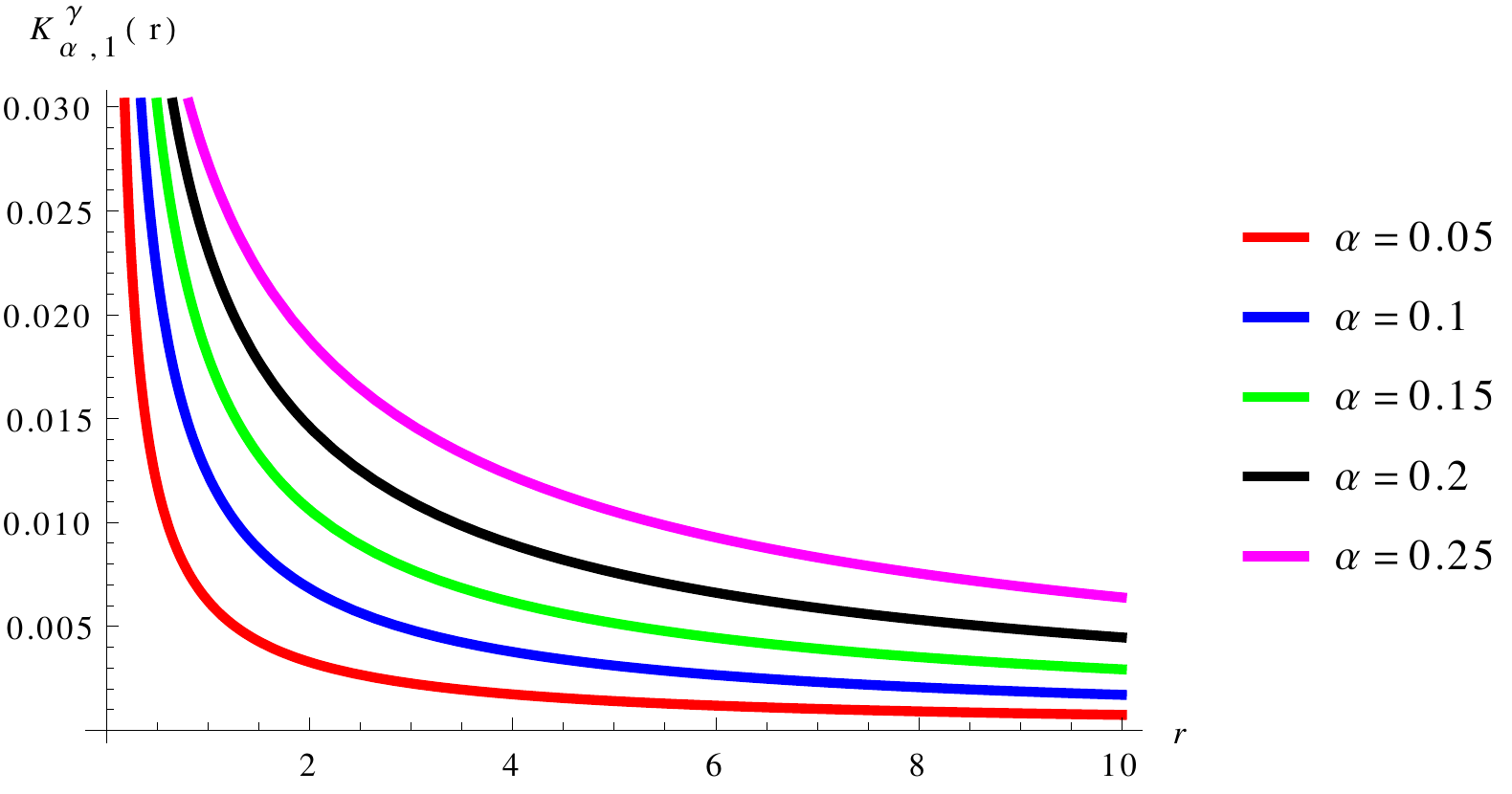}
			\includegraphics[scale=.42]{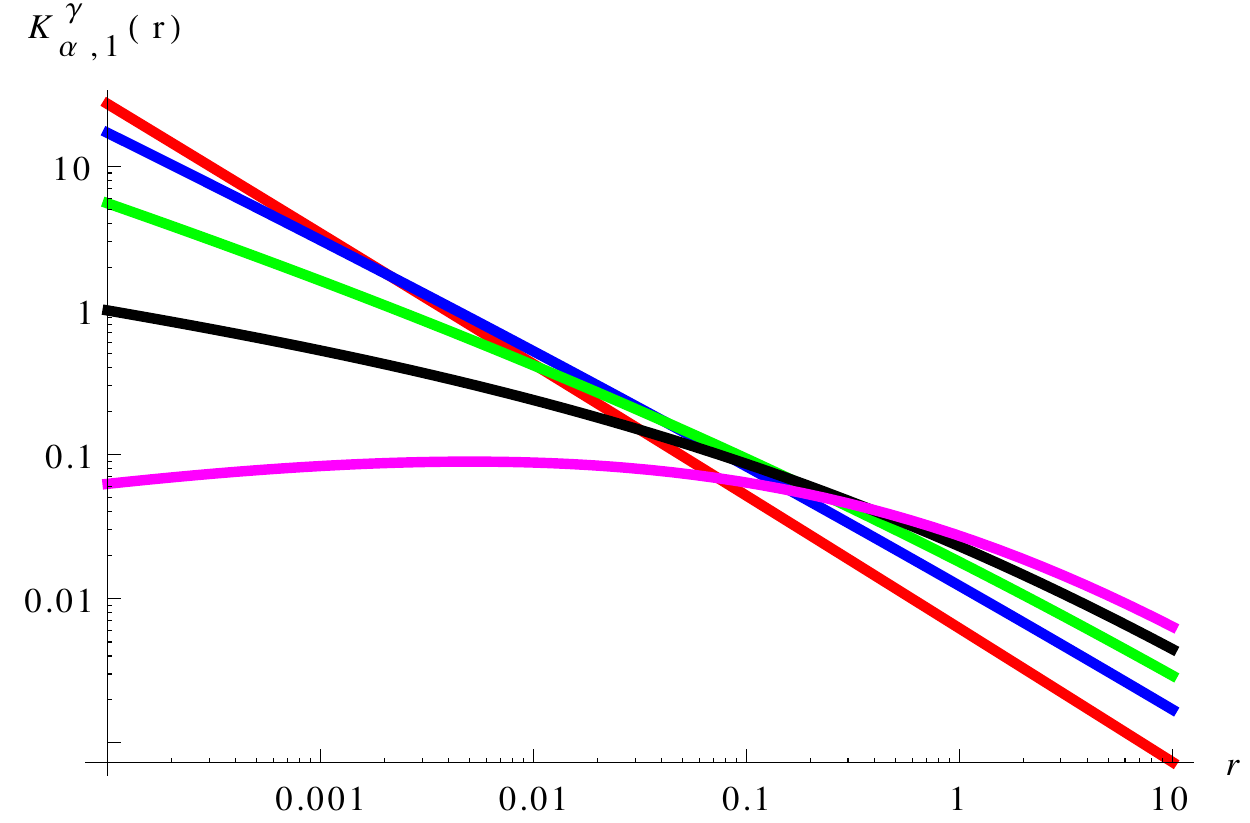}
			\caption{\label{bababa2}Plot (left) and log-log-plot (right)
			of the function \eqref{bababa0} for $\gamma=4$ and $\alpha=(0.05,0.1,0.15,0.2,0.25)$.}
		\end{figure}			

		\begin{figure}
			\centering
			\includegraphics[scale=.42]{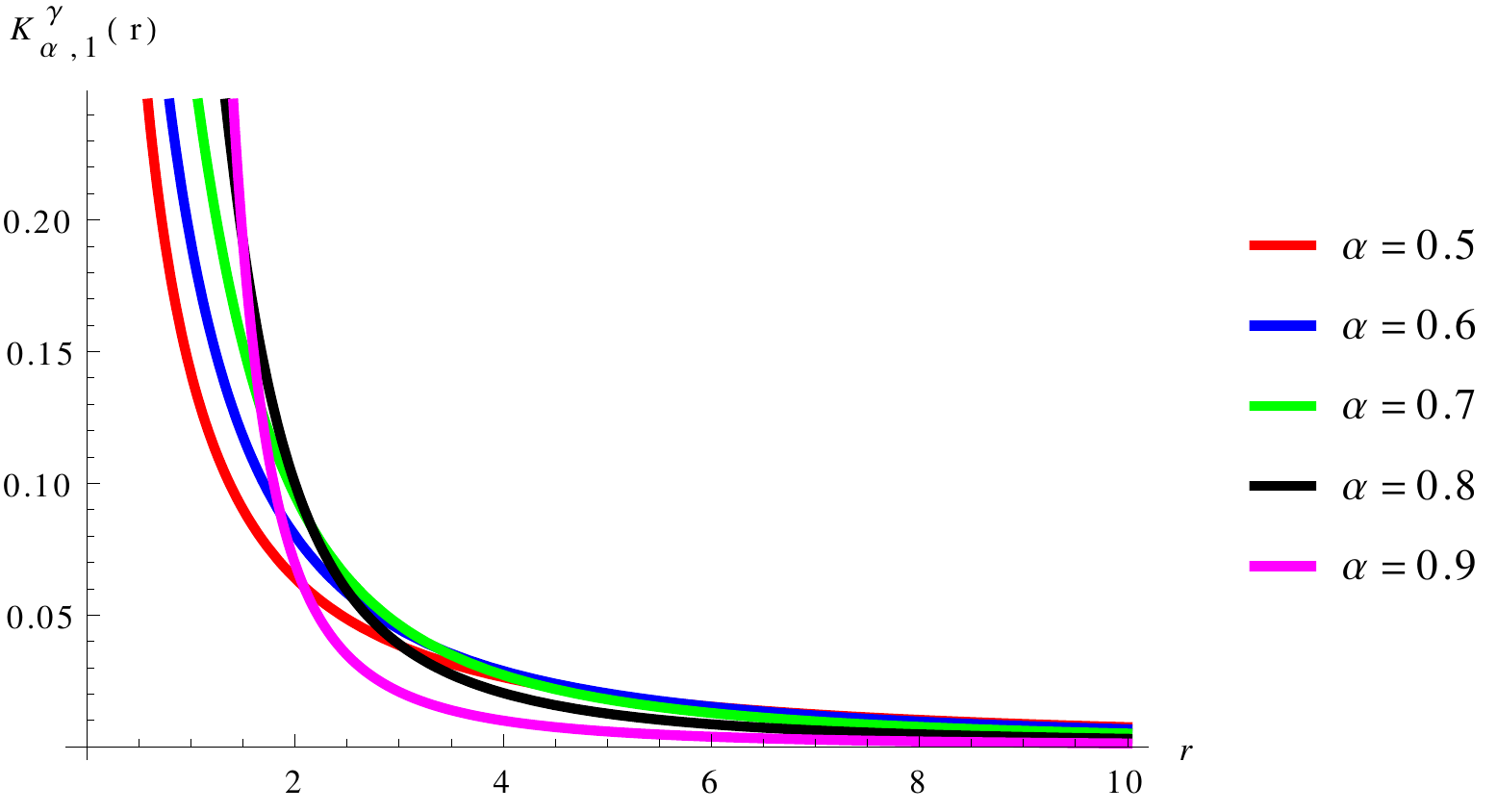}
			\includegraphics[scale=.42]{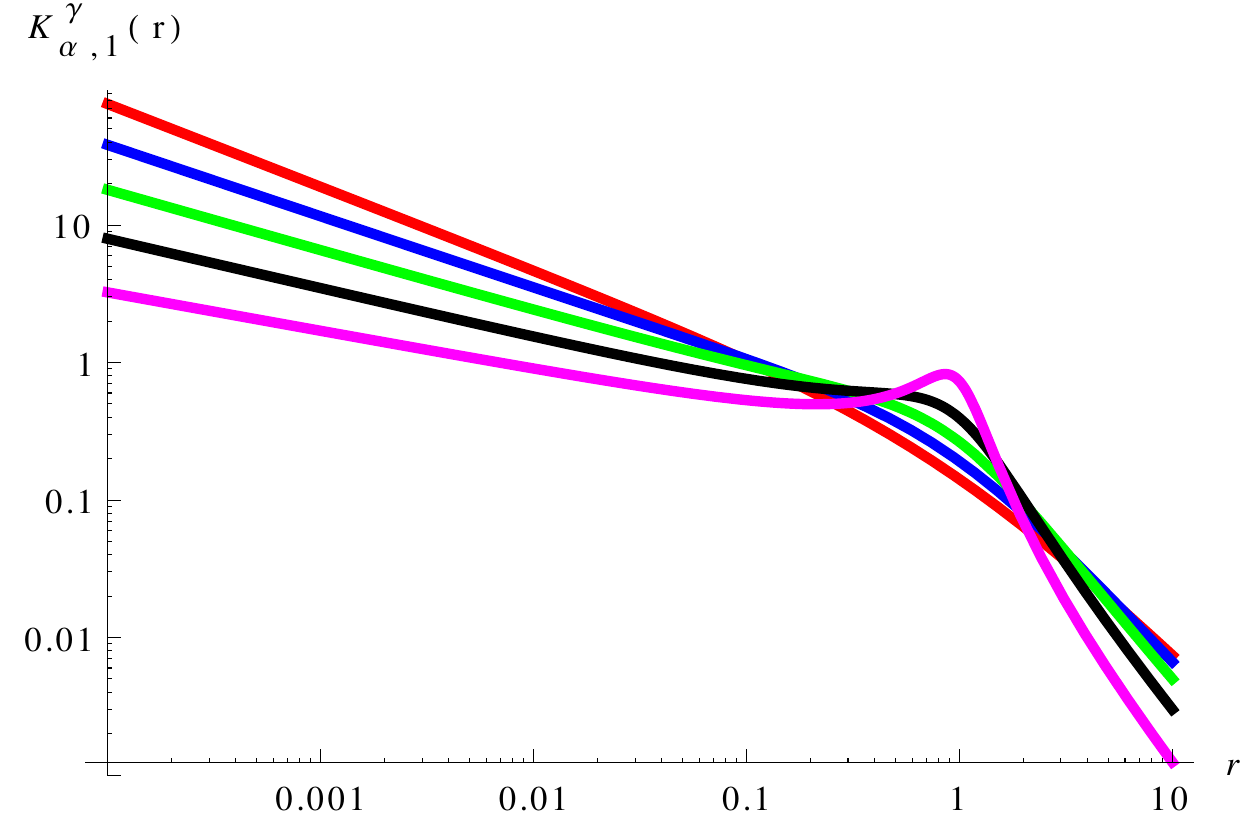}
			\caption{\label{bababa3}Plot (left) and log-log-plot (right)
			of the function \eqref{bababa0} for $\gamma=0.8$ and $\alpha=(0.5,0.6,0.7,0.8,0.9)$.}
		\end{figure}			

		We conclude by emphasizing that, if $ \alpha \in (0,1]$, $0< \alpha \gamma \le 1$, $r>0$, the kernel
		\begin{align}
			\label{bababa0}	
			K_{\alpha,1}^\gamma(r)
			= \frac{r^{\alpha \gamma -1}}{\pi} \frac{\sin \left( \gamma \arctan \left( \frac{r^\alpha \sin (\pi \alpha)}{
			r^\alpha \cos (\pi \alpha)+1} \right) + \pi (1-\alpha \gamma) \right)}{\left( r^{2\alpha} + 2r^\alpha
			\cos (\pi \alpha) +1 \right)^{\gamma/2}},
		\end{align}
		is the density
		of a probability measure concentrated on the positive real line (see Figures \ref{bababa}, \ref{bababa2} and \ref{bababa3}).

		\subsection*{Acknowledgements}
		
			\v{Z}ivorad Tomovski is supported by the European Commission and
			the Croatian Ministry of Science, Education and Sports Co-Financing Agreement
			No.\ 291823. In particular, \v{Z}ivorad Tomovski acknowledges the Marie Curie project FP7-PEOPLE-2011-COFUND
			program NEWFELPRO Grant Agreement No. 37 -- Anomalous diffusion.

	\bibliographystyle{abbrvnat}	
	\bibliography{paper}


\end{document}